\documentclass[10pt]{article}
\usepackage{lineno,hyperref}
\usepackage{amsmath,amssymb}
\usepackage{amssymb}
\usepackage{geometry}
\modulolinenumbers[5]

\numberwithin{equation}{section}

\newtheorem{Theorem}{Theorem}[section]

\newtheorem{Lemma}[Theorem]{Lemma}

\newtheorem{Remark}[Theorem]{Remark}

\newcommand{\cvd}{\hfill$\square$ \bigskip}

\begin{document}

%
%
%
%
%
%
%
%
%

\title {Deep estimates for the higher eigenvalues of the poly-Laplacian}

\author{\sc Zhengchao Ji \footnote{
Department of Mathematics, China Jiliang University,
             Hangzhou,  310027, People's Republic of China. E-mail: jizhengchao@zju.edu.cn} and Hongwei Xu  \footnote{
Center of Mathematical Sciences, Zhejiang University,
             Hangzhou,  310027, People's Republic of China. E-mail:xuhw@zju.edu.cn }}

\date{}
\maketitle




\begin{abstract}
We investigate the lower bound for higher eigenvalues $\lambda_i$ of the poly-Laplace operator on a bounded domain and improve the famous Li-Yau inequality and its related results. Firstly, we consider the low dimensional cases for the P\'{o}lya conjecture, the clamped plate problem and the eigenvalue problem of the poly-Laplacian and deliver a series of deep eigenvalue inequalities for these problems respectively. Secondly, we establish a sharp lower bound for the eigenvalues of the poly-Laplacia in arbitrary dimension under some certain restrictive conditions. Finally, we provide an improved inequality for $\lambda_i$ in arbitrary dimension without any restrictive conditions. Our results also yield the improvement of the lower bounds for the Stokes eigenvalue problems and the Generalized  P\'{o}lya conjecture.
\end{abstract}
\noindent{\bf keywords:} {higher eigenvalues, P\'{o}lya conjecture, the clamped plate problem,
poly-Laplacian}

\section{Introduction} 
In this paper, we study the following Dirichlet eigenvalue problem of the  poly-Laplace operator on a bounded Euclidean domain $\Omega\subseteq \mathbb{R}^n$,
\begin{equation}\label{EOE}
\begin{cases}
(-\Delta)^l u=\lambda^{(l)}  u  \ & \mathrm{in}\ \Omega, \\
u=\frac{\partial u}{\partial \nu}=\cdots=\frac{\partial^{l-1} u}{\partial \nu^{l-1}} = 0  \  & \mathrm{on}\ \partial\Omega.
\end{cases}
\end{equation}
The spectrums of this eigenvalue problem  are real and discrete (cf. \cite{AB,C,CQW,CSWZ,GGS,J,JX,JLWX,La,PPW,SY,Y})
\begin{alignat*}{1}
0<\lambda^{(l)}_1\leq \lambda^{(l)}_2\leq \lambda^{(l)}_3\leq \cdots\rightarrow\infty,
\end{alignat*}
where each $\lambda_i$ has finite multiplicity which is counted by its multiplicity. The Dirichlet eigenvalue problem of poly-Laplacian (such as fourth-order and sixth-order operators) has profound significance in physics, with applications spanning classical mechanics, quantum field theory, materials science, and cosmology. For example, higher eigenvalues of poly-Laplacian are related to the elimination of unstable higher-order modes in the theory (such as tachyon instability), verification of model feasibility in cosmology and gravitational wave propagation \cite{Ha}.

\subsection{Lower bounds for P\'{o}lya Conjecture}

Let $\emph{\textbf{\textit{l=1}}}$, then (\ref{EOE}) becomes the well-known \textbf{Dirichlet eigenvalue problem} and it governs the transverse vibration of a membrane with the same boundary condition. Let $V(\Omega)$ and $\omega_n$ denote the volumes of the bounded domain $\Omega$ and the unit ball in $\mathbb{R}^n$ respectively.  Weyl proved the following famous asymptotic formula for the Dirichlet eigenvalue as
\begin{alignat}{1}\label{waf}
\lambda^{(1)}_k\sim \frac{4\pi^2}{(\omega_nV(\Omega))^\frac{2}{n}}k^{\frac{2}{n}},\ k\rightarrow\infty.
\end{alignat}
Inspired by this one-term spectral asymptotics for the Dirichlet Laplacian, many celebrated inequalities were established  and we suggest  that   readers refer \cite{GMWW,KVW,LY,Li,LL,SY,Y} for  more details.

From Weyl's asymptotic formula, one can immediately infer that the sum of Dirichlet eigenvalues $\{\lambda_i\}_{i=1}^k$ satisfies
\begin{alignat}{1}\label{EAF}
\frac{1}{k}\sum_{i=1}^k \lambda^{(1)}_i \sim \frac{n}{n+2}\frac{4\pi^2}{(\omega_nV(\Omega))^\frac{2}{n}}k^{\frac{2}{n}},\ k\rightarrow\infty.
\end{alignat}
Based on the above asymptotic formula,  P\'{o}lya  proposed a famous open problem in 1954:

\textbf{P\'{o}lya Conjecture.} \textit{If $\Omega$ is a bounded domain in ${\mathbb{R}}^n$, then $k$-th eigenvalue $\lambda^{(1)}_k$ of the eigenvalue
problem (\ref{EOE}) for $l=1$  satisfies}
\begin{alignat*}{1}
\lambda^{(1)}_k \geq \frac{4\pi^2}{(\omega_nV(\Omega))^\frac{2}{n}}k^{\frac{2}{n}},\ \ \mathrm{for}\ k=1,2,\ldots.
\end{alignat*}

In 1961, P\'{o}lya \cite{Po} himself  confirmed this Conjecture on the \textit{tiling domain} in ${R}^2$. As this problem was proposed, it has received widespread attention in the pass six decades. Berezin \cite{Be} and Lieb \cite{Li,LL} gave a partial solution to this conjecture. In 1983, Li and Yau \cite{LY} verified the following remarkable Li-Yau eigenvalue inequality
\begin{alignat}{1}\label{LYE}
\frac{1}{k}\sum_{i=1}^k \lambda^{(1)}_i \geq \frac{n}{n+2}\frac{4\pi^2}{(\omega_nV(\Omega))^\frac{2}{n}}k^{\frac{2}{n}},\ k=1,2,\ldots.
\end{alignat}
According to (\ref{EAF}), Li-Yau's inequality  is the best possible in the sense of the
average of eigenvalues. Thus, (\ref{LYE}) gives a partial solution to the P\'{o}lya conjecture  with a factor $\frac{n}{n+2}$ and this factor has not been improved up to now. In 2003, Melas \cite{M} first introduced $I(\Omega)$ (the moment of \textit{inertia} of $\Omega$) to estimate the lower bound for the Dirichlet eigenvalues and he  obtained the following beautiful estimate which improves (\ref{LYE}) as
\begin{alignat}{1}\label{MLE}
\sum_{i=1}^k \lambda^{(1)}_i \geq   \frac{n}{n+2}\frac{4\pi^2}{(\omega_nV(\Omega))^\frac{2}{n}}k^{\frac{n+2}{n}}+ \tilde{c}_n\frac{V(\Omega)}{I(\Omega)}k,
\end{alignat}
where $\tilde{c}_n=1/(24(n+2))$. In 2009, Kova\v{r}\'{i}k, Vugalter and Weidl \cite{KVW} improved (\ref{MLE}) on the two-dimensional bounded domains $\Omega\subseteq\mathbb{R}^2$. In 2010, Ilyin \cite{Il} improved Melas's lower bound when $n=2,3,4$ and he proved that
\begin{alignat}{1}\label{Ilsl2}
\sum_{i=1}^k\lambda^{(1)}_i\geq \frac{n}{n+2}\frac{4\pi^2}{(\omega_nV(\Omega))^\frac{2}{n}}k^{1+\frac{2}{n}} +\beta_n^L\frac{n}{48}\frac{V(\Omega)}{I(\Omega)}k
\end{alignat}
where $\beta_2^L= 119/120,\beta_3^L=0.986,\beta_4^L=0.978$ for $n=2,3,4$ respectively.

In 2013,  Yildirim and Yolcu  \cite{YY2} improved (\ref{MLE}) by adding the following  term on the right hand of  (\ref{MLE})
\begin{alignat}{1}\label{YYML}
\begin{split}
\frac{V(\Omega)^{\frac{3n+2}{2n}}}{144(n+2)I^{\frac{3}{2}}\Gamma^{\frac{1}{n}}(1+n/2)}k^{1-\frac{1}{n}}.
\end{split}
\end{alignat}

In 2020, the authors of this paper \cite{JX} proved a sharp polynomial inequality and also improved (\ref{YYML}) for any $n\geq 2$ and $k\geq 1$ by adding more positive terms on the right hand of (\ref{MLE}). For example, by using results in \cite{JX}, we can sharp (\ref{MLE}) as
\begin{alignat}{1}\label{JX1}
\begin{split}
\sum_{i=1}^k\lambda^{(1)}_i\geq & \frac{n{\omega_n}^{-\frac{2}{n}}(2\pi)^2{V}^{-\frac{2}{n}}(\Omega)}{n+2}k^{\frac{n+2}{n}}
+ \tilde{c}_n\left(\frac{V(\Omega)}{I(\Omega)}\right)k +\frac{\omega_n^{-\frac{1}{n}}\alpha^{\frac{3n+1}{n}}}{9(n+2)\rho^3}k^{\frac{n-1}{n}},
\end{split}
\end{alignat}
where
\begin{alignat}{1}
\alpha= {V(\Omega)}/((2\pi)^n),\,\,\,\rho=2(2\pi)^{-n}\sqrt{V(\Omega)I(\Omega)}.
\end{alignat}

Recently, Filonov, Levitin, Polterovich and Sher \cite{FLPS} achieved a breakthrough work, they solved P\'{o}lya's Conjecture  for the disk, making it the first non-tiling planar domain for which the conjecture is verified. In the same paper, they  also confirmed P\'{o}lya's conjecture for arbitrary planar sectors, and, in the Dirichlet case, for balls of any dimension.  A related interesting  problem is to investigate the  P\'{o}lya Conjecture in a Cartan-Hadamard manifold and we recommend readers to refer to \cite{XX,X1} for details.

\subsection{Lower bounds for the clamped plate problem}

Let $\emph{\textbf{\textit{l=2}}}$, then the equation (\ref{EOE})  governs the motion of the vibration of a stiff plate and becomes the \textbf{clamped plate problem}.  An important eigenvalue problem for the lower bound  of the clamped plate problem is the famous \textbf{Lord Rayleigh's conjecture} and we recommend readers to refer to \cite{AB2,K1,K2,Le} for details.

 For the eigenvalues of the clamped plate problem, Agmon \cite{A} and Pleijel \cite{Pl}  gave the following asymptotic formula
\begin{alignat}{1}\label{CPAF}
\lambda^{(2)}_k \sim \frac{16\pi^4}{(\omega_n V(\Omega))^{\frac{4}{n}}}k^{\frac{4}{n}},\ k\rightarrow\infty.
\end{alignat}
In 1985, Levine and Protter \cite{LP} got a lower bound for $\sum^k_{i=1}\lambda^{(2)}_i$ as
\begin{alignat*}{1}
\frac{1}{k}\sum_{i=1}^k\lambda^{(2)}_i \geq \frac{n}{n+4} \frac{16\pi^4}{(\omega_n V(\Omega))^{\frac{4}{n}}}k^{\frac{4}{n}}.
\end{alignat*}
The formula (\ref{CPAF}) implies that the coefficient of $k^{\frac{4}{n}}$ on the right hand of the above inequality is the best possible in the sense of the average of eigenvalues.

Then, Cheng and Wei \cite{CW1,CW2} improved Levine-Protter's lower bound. In \cite{YY1}, Yildirim and Yolcu improved Cheng and Wei's eigenvalue inequality when $n\geq 2$ and $ k \geq 1$. In 2013, Cheng, Sun, Wei and Zeng \cite{CSWZ} obtained a similar lower bound for $\sum_{i=1}^k \lambda^{(2)}_i$ as (\ref{Ilsl2}) when $n=2,3,4$, they got
\begin{alignat}{1}\label{CSWZld}
\frac{1}{k}\sum_{i=1}^k\lambda^{(2)}_i \geq \frac{n}{n+4} \frac{16\pi^4}{(\omega_n V(\Omega))^{\frac{4}{n}}}k^{\frac{4}{n}}+\alpha_n\frac{n}{24}\frac{V}{I}\frac{(2\pi)^2}{(V\omega_n)^{2/n}}k^{\frac{2}{n}},
\end{alignat}
where $0<\alpha_n<1$.

By proving  a sharp polynomial inequality,  authors \cite{JX}  improved  Yildirim-Yolcu' estimate for any $n\geq 3$ and $k\geq 1$ by adding more positive terms. For instance, the corollary in \cite{JX} provides the following lower bound
\begin{alignat}{1}\label{JX2}
\begin{split}
\sum_{i=1}^k \lambda^{(2)}_i \geq&\frac{n}{n+4}(\omega_n)^{-\frac{4}{n}}\alpha^{-\frac{4}{n}}k^{\frac{4+n}{n}}
+\frac{1}{3(n+4)}\frac{(\omega_n)^{-\frac{2}{n}}\alpha^{\frac{2n-2}{n}}k^{\frac{n+2}{n}}}{\rho^2}\\
& +\frac{2}{9(n+4)}\frac{(\omega_n)^{-\frac{1}{n}}\alpha^{\frac{3n-1}{n}}k^{\frac{n+1}{n}}}{\rho^3}+\frac{3\alpha^{4}}{40(n+4)\rho^4}k.
\end{split}
\end{alignat}

\subsection{Lower bounds for the poly-Laplacian}

For any integer $\emph{\textbf{\textit{l}}}$, the eigenvalue problem (\ref{EOE}) is \textbf{the most important case} of the generalized P\'{o}lya Conjecture. Let  $t$ be a fixed positive integer, Ku-Ku-Tang \cite{KKT} considered the Dirichlet eigenvalue problem of a elliptic operator $\mathfrak{L}$ which satisfies $\mathfrak{L}u=\sum^{t}_{m=r+1}a_{m-r}(-\Delta)^{r}u$, where $r\geq0$ is an integer and $a_m's$ are nonnegative constants with $r+1\leq m\leq t$. In  \cite{KKT}, Ku-Ku-Tang proposed \textbf{ the generalized P\'{o}lya Conjecture} as: The Dirichlet eigenvalues $\lambda_{k,r}$ of $\mathfrak{L}$ satisfies the inequalities
\begin{alignat*}{1}
\lambda_{k,r}\geq \sum^{t}_{m=r+1}a_{m}\left(\frac{4\pi^2}{\omega^{\frac{2}{n}}_n} \right)^m\left({\frac{k}{V}}  \right)^{\frac{2m}{n}}.
\end{alignat*}
We should mention that the most important case of the generalized P\'{o}lya Conjecture is the eigenvalue problem (\ref{EOE}). In fact, by using the discussion in Section 2, we can prove that
\begin{alignat}{1}\label{gpc2}
\sum_{i=1}^k\lambda_{k,r}\geq n\omega_n\sum^{t}_{m=r+1}a_{m}\int_0^\infty s^{2r+n-1}\phi(s)ds,
\end{alignat}
where $\phi(s)$ is a certain function relating to the eigenfunctions of $\mathfrak{L}$. Hence, our lower bounds for $(\ref{LE})$ is the key estimate for \textbf{ the generalized P\'{o}lya Conjecture}.

There are many important eigenvalue inequalities for $\lambda_i^{l}$ were obtained in \cite{GGS,H,LP,LR}. And Cheng-Qi-Wei \cite{CQW} proved that
\begin{alignat}{1}\label{cqw}
\begin{split}
\frac{1}{k}\sum_{i=1}^k\lambda^{(l)}
\geq & \frac{n}{n+2l}\frac{(2\pi)^{2l}}{(\omega_nV(\Omega))^{\frac{2l}{n}}}k^{\frac{2l}{n}}+ \frac{n}{n+2l}\sum^l_{p=1}\tilde{c}_{n,p}\left( \frac{V(\Omega)}{I(\Omega)}\right)^{p}k^{\frac{2(l-p)}{n}},
\end{split}
\end{alignat}
where
\begin{alignat*}{1}
\tilde{c}_{n,p}= \frac{(2\pi)^{2(l-p)}(l+1-p)}{ 24^p n\cdots(n+2p-2)(\omega_nV)^{2(l-p)/n}}.
\end{alignat*}

In this paper, we will prove a shaper lower bounds of $\sum_{i=1}^k \lambda^{(l)}_i$ for bounded $k$ and improve the results in \cite{CQW,CSWZ}. In addition, we consider several lower dimensional cases and obtain shaper eigenvalue inequalities respectively.

\subsection{Ways to obtain the lower bounds}
Before we state our main results, we should discuss the following ways to obtain the lower bounds for the P\'{o}lya conjecture.  \textbf{{Li-Yau's way:}} Li-Yau \cite{LY} originally used the Fourier transform to covert the estimation of $\sum_{i=1}^k\lambda^{(1)}$ to the upper bound $M_1$ of some certain decreasing radial rearrangement function $F_1^{*}(|\xi|)$. They proved the following key inequality
\begin{alignat*}{1}
\frac{n}{n+2}\left(\frac{n}{M_1\omega_{n-1}} \right)^\frac{2}{n}\left(\int_{R^n} F^*_1(|\xi|)d\xi \right)^{\frac{n+2}{n}}\leq M_2,
\end{alignat*}
where $M_1$ is the upper bound of $F^*_1$ and $M_2$ is the upper bound of $\int_{R^n}|\xi|^2F^*_1(|\xi|)d\xi$. Using the properties of the Fourier series and the rearrangement inequality, Li-Yau proved the best value of  $M_1$ is  $\alpha$, then they obtain (\ref{LYE}).  \textbf{{Males's way:}} Males \cite{M} noticed that
\begin{alignat*}{1}
\int_{R^n}|\xi|^2F^*_1(|\xi|)d\xi\leq \sum_{i=1}^k\lambda^{(1)}.
\end{alignat*}
Moreover, he proved  that the upper bound for the derivative of $F^*_1(|\xi|)$ is $\rho$.  Then, Males used\textbf{ a three-term binary polynomial} to obtain the lower bound for $\int_{R^n}|\xi|^2F^*_1(|\xi|)d\xi$ and obtain (\ref{MLE}). \textbf{Ilyin's way:} In \cite{Il}, Ilyin found the minimizer $\Phi_s(\xi)$ of the critical integral $\int_{R^n}|\xi|^2F^*_1(|\xi|)d\xi$ under the restriction of $F^*_1\leq \alpha$ and $(F^*_1)'\leq \rho$. However, finding the lower bound bound  for $\int_{R^n}|\xi|^2\Phi_s(|\xi|)d\xi$  is equivalent to finding the lower bound of the root of\textbf{ a higher-order equation}. Ilyin solved this equation when $n=2,3,4$.  Consequently, he proved (\ref{Ilsl2}). All of these three ways are also applicable for the Generalized  P\'{o}lya conjecture.

In this paper, we obtain the lower bound for  Ilyin's critical equation for arbitrary dimension under certain restricted condition on $k$. Therefore, we prove a sharper lower bounds for $\sum_{i=1}^k\lambda^{(l)}$. Moreover, we  also find a $m-$term ($\forall m\leq n$)  binary polynomial to obtain the lower bound for $\int_{R^n}|\xi|^{2l}F^*_1(|\xi|)d\xi$, then we obtain a shaper bound for $\sum_{i=1}^k\lambda^{(l)}$ without any extra condition on $k$.

\subsection{Main results}
To summarize, we will mainly use two different methods to establish  sharper lower bounds of  $\sum_{i=1}^k \lambda^{(l)}_i$  for arbitrary $l\geq 1$.

In Section 3, we will improve (\ref{Ilsl2}) by replacing $\beta^L_n$ by 1 for $l=1$ when $n=3,4$. Moreover, we improve  Cheng-Sun-Wei-Zeng's lower bounds (\ref{CSWZld}) by replacing $\beta^L_n$ by 1 for $l=2$  when $n=2,3,4$. For $l\geq 3$, we also establish several \textbf{new eigenvalue inequalities} which are sharper than (\ref{cqw}) when $n=2,3,4$.

In Section 4, we provide new lower bounds for $\sum_{i=1}^k\lambda^{(l)}_i$ when $n\geq 5$ with certain restriction on $k$ (see Theorem \ref{MT11}). For $l=1$, we replace $\tilde{c}_n$ in (\ref{MLE})into $\frac{n}{6}$. For $l=2$, we change $\frac{1}{3(n+4)}$ in (\ref{JX2}) into $\frac{n}{6}$. For $l\geq 3$, we improve the coefficient of $\omega_n^{-\frac{2l-2}{n}}\alpha^{\frac{2l-2n-2}{n}}\rho^{-2}k^{\frac{2l+n-2}{n}}$ in (\ref{cqw}) from $\frac{l}{ 24(n+2l) }$ into
$\frac{nl}{12}$ (see Theorem \ref{MT11}).

In Section 5, we provide sharper lower bounds for $\sum_{i=1}^k\lambda^{(l)}_i$ when $2l+n\geq 6$ without any restriction on $k$ (see Theorem \ref{LWBFAk}). For $l=1$, we replace $\tilde{c}_n$ in (\ref{MLE})into $\frac{5}{2(n+2)}$. For $l=2$, we change $\frac{1}{3(n+4)}$ in (\ref{JX2}) into $\frac{5}{(n+4)}$. For $l\geq 3$, we improve the coefficient of $\omega_n^{-\frac{2l-2}{n}}\alpha^{\frac{2l-2n-2}{n}}\rho^{-2}k^{\frac{2l+n-2}{n}}$ in (\ref{cqw}) from $\frac{l}{ 24(n+2l) }$ into
$\frac{5l}{2(2l+n)}$ (see Theorem \ref{LWBFAk}).

Specifically, we will deliver an \textbf{optimal}  inequality (see Lemma \ref{KLL}) and several \textbf{optimal} lower bounds for the positive root to the critical equation  $(t+1)^{n+1}-t^{n+1}=Q$, then use these inequalities to estimate the higher eigenvalues of (\ref{EOE}). For general $n\geq 5$, we will prove the following  improved lower bound for  $\sum_{i=1}^k \lambda^{(l)}_i$ with restriction on $k$.  This inequality is sharper than the results in  \cite{CQW,CSWZ}. In section 4, we will give the following estimation.

\begin{Theorem}\label{MT11}
For any bounded domain $\Omega\subseteq \mathbb{R}^n$ ($n\geq 5$), let $\{\lambda^{(l)}_i\}$ be the eigenvalues  of the eigenvalue problem (\ref{EOE}).
If  $k\geq \max\{\tilde{k}^{n/2}_1,\tilde{k}^{n/2}_2,\tilde{k}_3,\tilde{k}_4\}$, where $\tilde{k}_1,\tilde{k}_2,\tilde{k}_3$ and $\tilde{k}_4$ are defined by (\ref{ng7k1}),(\ref{ng7k2}),(\ref{gnglk3}) and (\ref{gnglk4}),  then
\begin{alignat*}{1}
\begin{split}
\sum_{i=1}^k\lambda^{(l)}_j(\Omega)\geq& \frac{n}{n+2l}\frac{(2\pi)^{2l}}{(\omega_nV(\Omega))^{\frac{2l}{n}}}k^{\frac{2l+n}{n}}+\frac{nl}{12\rho^{2}}\omega_n^{-\frac{2l-2}{n}}\alpha^{\frac{2l-2n-2}{n}}k^{\frac{2l+n-2}{n}}\\
&+ \tilde{C}_1k^{\frac{2l+n-3}{n}}+ \tilde{C}_2 k^{\frac{2l+n-4}{n}}+\tilde{C}_3k^{\frac{2l+n-5}{n}},
\end{split}
\end{alignat*}
where
\begin{alignat*}{1}
\begin{split}
\tilde{C}_1&=-\frac{n(12l+7n-13)(2l+n-1)}{48\rho^3\omega_n^{\frac{2l-3}{n}}}\alpha^{\frac{-2l+3n+3}{n}},\\
\tilde{C}_2&=+\frac{n(2l+2n-3)(2l+n-1)(2l+n-2)}{96\rho^4\omega_n^{\frac{2l-4}{n}}}\alpha^{\frac{-2l+4n+4}{n}},\\
\tilde{C}_3&=-\frac{n(2l+n-1)(2l+n-2)(2l+n-3)(2l+n-4)}{288\rho^5\omega_n^{\frac{2l-5}{n}}}\alpha^{\frac{-2l+5n+5}{n}}.\\
\end{split}
\end{alignat*}

\end{Theorem}

In  section 5, we will give a sharper inequality for $\sum_{i=1}^k \lambda^{(l)}_i$ without any limitations. This lower bound will improve the eigenvalue inequalities in \cite{CQW,JX}. Specifically, we will prove the following estimation.

\begin{Theorem}\label{LWBFAk}
For any bounded domain $\Omega\subseteq \mathbb{R}^n$, let $\{\lambda^{(l)}_i\}$ be the eigenvalues  of the eigenvalue problem (\ref{EOE}).
If  $2l+n\geq 6$,  then
\begin{alignat*}{1}
\sum_{i=1}^k\lambda^{(l)}_i \geq&\frac{n}{n+2l}\frac{(2\pi)^{2l}}{(\omega_nV(\Omega))^{\frac{2l}{n}}}k^{\frac{2l+n}{n}}+\frac{5l}{2 (2l+n)\rho^2}\omega_n^{-\frac{2l-2}{n}}\alpha^{\frac{2n-2l+2}{n}}k^{\frac{2l+n-2}{n}}\\
&-\frac{31l\omega_n\alpha^{\frac{3n-2l+3}{n}}}{9(2l+n)\rho^3}\left(\frac{k}{\omega_n} \right)^{\frac{2l+n-3}{n}}
+\frac{5l\omega_n\alpha^{\frac{4n-2l+4}{n}}}{8(2l+n)\rho^4}\left(\frac{k}{\omega_n} \right)^{\frac{2l+n-4}{n}}\\
&+\frac{38l\omega_n\alpha^{\frac{5n-2l+5}{n}}}{25(2l+n)\rho^5}\left(\frac{k}{\omega_n} \right)^{\frac{2l+n-6}{n}}
-\frac{317l\omega_n\alpha^{\frac{6n-2l+6}{n}}}{420(2l+n)\rho^6}\left(\frac{k}{\omega_n} \right)^{\frac{2l+n-6}{n}}.
\end{alignat*}

\end{Theorem}

\begin{Remark}
According to the similar discussion in \cite{Il},  when $l=1$, if we replace $\alpha$ by $(n-1)(2\pi)^{-n}V$ and replace $\rho$  by  $(n(n-1))^{1/2}(2\pi)^{-n}I(\Omega)$, then the lower bounds for $\sum_{i=1}^k\lambda^{(1)}_i$ is also true for $\sum_{i=1}^k\lambda^{s}_i$, where $\{\lambda^{s}_i\}_{i=1}^k$ are the Dirichlet eigenvalues of the Stokes operator. And these lower bounds also improve the lower bounds of $\{\lambda^{s}_i\}_{i=1}^k$
in \cite{Il}.

Moreover, according to the discussion of (\ref{gpc2}), by using our results, one can immediately obtain sharp lower bounds for  $\sum_{i=1}^k\lambda_{k,r}$.
\end{Remark}

\section{Preliminaries}

In this section, we  introduce some important notations and definitions.  For a bounded domain $\Omega$ in $\mathbb{R}^{n}$, \textit{the moment of inertia} of $\Omega$ is defined as $I(\Omega)=\int_{\Omega} |x|^2 dx.$ For convenience,  we use $I,V$ to denote $I(\Omega),V(\Omega)$  and $C^m_n=n!/(m!(n-m)!)$. Obviously,
\begin{alignat*}{1}
I(\Omega)\geq \frac{n}{n+2}V(\Omega)\left(\frac{V(\Omega)}{\omega_n} \right)^{\frac{2}{n}}.
\end{alignat*}

Without ambiguity, we will use $\lambda_i$ to denote $\lambda^{(l)}_i$ in the rest of this paper.   We fixed a $k\geq 1$ and let $\{u_j\}_{i=1}^{k}$ be the orthonormal eigenfunctions corresponding to the eigenvalues $\{\lambda_j\}_{j=1}^k$ for the following Dirichlet problem of the poly-Laplacian
\begin{equation}\label{EOE1}
\begin{cases}
(-\Delta)^l u_j=\lambda _ju_j  \ & \mathrm{in}\ \Omega, \\
u=\frac{\partial u}{\partial \nu}=\cdots=\frac{\partial^{l-1} u}{\partial \nu^{l-1}} = 0  \  & \mathrm{on}\ \partial\Omega.
\end{cases}
\end{equation}

We consider the Fourier transform of each eigenfunction
\begin{alignat}{1}
f_j(\xi)=\hat{u}_j(\xi)=(2\pi)^{-n/2}\int_{\Omega} u_j(x)e^{ix\cdot\xi}dx.
\end{alignat}
From Plancherel's Theorem, we know that  $f_1 ,.\ldots, f_k$ is an orthonormal set in $\mathbb{R}^n$, which impies
\begin{alignat*}{1}
\int_{\Omega}f_i(\xi)f_j(\xi)dx=
\begin{cases}
0,  \ & \mathrm{if}\ i\neq j, \\
1,  \  & \mathrm{if}\ i=j.
\end{cases}
\end{alignat*}

Since these eigenfunctions $u_1 ,\ldots,u_k$ are also orthonormal in $L_2(\Omega)$, Bessel's inequality implies that for every $\xi\in \mathbb{R}^n$
\begin{alignat}{1}\label{S1}
\sum_{j=1}^k |f_j(\xi)|^2\leq (2\pi)^{-n}\int_{\Omega} |e^{ix\cdot\xi}|^2dx=(2\pi)^{-n}V(\Omega)
\end{alignat}
and since
\begin{alignat}{1}
\nabla f_j(\xi)=(2\pi)^{-n/2}\int_{\Omega}ixu_j(x)e^{ix\cdot\xi}dx,
\end{alignat}
then
\begin{alignat}{1}\label{1S2}
\sum_{j=1}^k |\nabla f_j(\xi)|^{2}\leq (2\pi)^{-n/2}\int_{\Omega}|ixe^{ix\xi}|^2dx=(2\pi)^{-n}VI.
\end{alignat}

Due to the boundary condition, it is easy to get
\begin{alignat}{1}\label{uod11111}
\int_{R^n}|\xi|^{2l}|f_j(\xi)|^2d\xi=\int_{\Omega}|\nabla u_j (x)|^{2l}dx=\lambda_j(\Omega)
\end{alignat}
for each $j$.

By using integration by parts, we have
\begin{alignat*}{1}
\sum_{i=1}^k\lambda _i=&\sum_{i=1}^k\int_{\Omega}u_i(x)(-\Delta)^{l} u_i(x)dx=\sum_{i=1}^k\sum_{j_1,\cdots,j_l=1}^n\int_{\mathbb{R}^n}\bigg( \frac{\partial^lu_{i}(x)}{\partial x_{j_1}\cdots\partial x_{j_l}} \bigg)^2dx.
\end{alignat*}
According to the Fourier transform and the Parseval's indentity, we can rewrite the above equation as
\begin{alignat}{1}\label{1le}
\begin{split}
\sum_{i=1}^k\lambda _i=&\sum_{i=1}^k\sum_{j_1,\cdots,j_l=1}^n\int_{\mathbb{R}^n}\bigg|(2\pi)^{-\frac{n}{2}}\int_{\Omega}\xi_{j_1}\cdots\xi_{j_l}u_i(x)e^{ix\cdot\xi}dx \bigg|^2d\xi.
\end{split}
\end{alignat}

Set $F(\xi)=\sum_{j=1}^k|f_j(\xi)|^2$. From (\ref{S1})-(\ref{1S2}), we have $0\leq F(\xi)\leq (2\pi)^{-n}V(\Omega)$
and
\begin{alignat}{1}
\begin{split}
|\nabla F(\xi)|\leq& 2(2\pi)^{-n}\sqrt{V(\Omega)I(\Omega)}\label{T1}
\end{split}
\end{alignat}
for each $\xi\in \mathbb{R}^n$.  Combing the definition of $F(\xi)$ with (\ref{1le}), we obtain the following equations
\begin{alignat}{1}
\int_{R^n}F(\xi)d\xi&=k,\\
\int_{R^n} |\xi|^{2l}F(\xi)d\xi&=\sum_{j=1}^k \lambda_j(\Omega).\label{T2}
\end{alignat}

Next, we will introduce the decreasing radial rearrangement function. For convenience, we adhere the same notations in \cite{JX}.  Let $F^*(\xi) = \phi (|\xi|)$ denote the decreasing radial rearrangement of $F(\xi)$. By approximating $F(\xi)$, we assume that  $\phi: [0,+\infty)\rightarrow [0,(2\pi)^{-n}V(\Omega)]$  is a absolutely continuous and  decreasing function. Setting $\mu(t)=|\{F^*>t \}|=|\{F>t\}|$ and the co-area formula implies that
\begin{alignat}{1}\label{}
\mu(t)=\int_{t}^{(2\pi)^{-n}V(\Omega)} \int_{\{F=s \}}\frac{1}{|\nabla F|}d\sigma_sds
\end{alignat}
Since $F^*(\xi)$ is a radial function, then we have $\mu(\phi(s))=|\{F^*>\phi(s) \}|=\omega_n s^n$. Moreover,  we can also get $n\omega_n s^{n-1}={\mu}{'}(\phi(s))\phi{'}(s)$ for almost every $s$. According to (\ref{T1}), the isoperimetric inequality and $\rho=2(2\pi)^{-n}\sqrt{V(\Omega)I(\Omega)}$, there holds
\begin{alignat*}{1}
\begin{split}
-\mu{'}(\phi(s)) &=\int_{\{F=\phi(s) \}} |\nabla F|^{-1} d\sigma_{\phi(s)}\geq \rho^{-1}\mathrm{Vol}_{n-1}(\{F=\phi(s)\})\geq \rho^{-1}n\omega_n s^{n-1}.
\end{split}
\end{alignat*}
Therefore, we arrive at
\begin{alignat}{1}
-\rho\leq \phi^{'}(s)\leq 0
\end{alignat}
for almost every $s$.

Notice that (\ref{T2}) implies that
\begin{alignat}{1}
k=\int_{R^n}F(\xi)d\xi=\int_{R^n}F^*(\xi)d\xi=n\omega_n\int_0^{\infty}s^{n-1}\phi(s)ds
\end{alignat}
and
\begin{alignat}{1}\label{LE}
\sum_{i=1}^k\lambda_j(\Omega)=\int_{R^n}|\xi|^{2l}F(\xi)d\xi\geq \int_{R^n}|\xi|^{2l}F^*(\xi)d\xi=n\omega_n\int_0^\infty s^{2l+n-1}\phi(s)ds
\end{alignat}
since $\xi \rightarrow |\xi|^2$ is radial and increasing. Before we introduce the next lemma, we should define the following function
\begin{alignat*}{1}
\Phi_s(r)dx=
\begin{cases}
M,  \ & \mathrm{for}\ \,\, 0\leq r\leq s, \\
M-L(r-s),  \  & \mathrm{for}\ \,\,s\leq r\leq s+\frac{M}{L},\\
0,  \  & \mathrm{for}\ \,\,s+\frac{M}{L}\leq r.
\end{cases}
\end{alignat*}

The following lemma is established in \cite{Il}.

\begin{Lemma}
Let $\alpha>0$ be a positive number and $\Psi(r)$ be a decreasing and absolutely continuous function and satisfy the following condition $\int_{0}^{\infty} r^{\alpha}\Psi(r)dr=m^{*}$ and $0\leq \Psi\leq M$,$-L<\Psi'<0$,
where $m^{*}=\int_{0}^{\infty} r^{\alpha}\Phi_s(r)dr$, then the following inequality holds
\begin{alignat*}{1}
\int_{0}^{\infty}r^{\beta}\Psi(r)dr\geq\int_{0}^{\infty}r^{\beta}\Phi_s(r)dr,
\end{alignat*}
for any $\beta\geq \alpha$.
\end{Lemma}

A straightforward calculation gives the  following integral identity
\begin{alignat*}{1}
\int_{0}^{\infty}r^{\gamma}\Phi_s(r)dr=\frac{M^{\gamma+2}}{(\gamma+1)(\gamma+2)L^{\gamma+1}}\left((t+1)^{\gamma+2} -t^{\gamma+2} \right),
\end{alignat*}
where $s=tM/L$. From the definition of $\phi(s)$, we get $0\leq \phi\leq (2\pi)^{-n}V(\Omega)$ and $\int_0^{\infty}s^{n-1}\phi(s)ds= k/(n\omega_n).$
Hence, let $\gamma=n-1$ and $t$ be the root of $(t+1)^{n+1} -t^{n+1}=Q$,
where
\begin{alignat*}{1}
Q:=\frac{(n+1)\rho^n}{\left((2\pi)^{-n}V(\Omega)  \right)^{n+1}}\frac{k}{\omega_n},
\end{alignat*}
we  derive that $\int_{0}^{\infty}s^{\beta}\phi(s)ds\geq \int_{0}^{\infty}r^{\beta}\Phi_s(r)dr.$ This inequality implies that
\begin{alignat*}{1}
\int_{0}^{\infty}s^{\beta}\phi(s)ds\geq\frac{\left((2\pi)^{-n}V(\Omega) \right)^{\beta+2}}{(\beta+1)(\beta+2)\rho^{\beta+1}}\left((t+1)^{\beta+2} -t^{\beta+2} \right)
\end{alignat*}
for any $\beta\geq n-1$. Combining the above inequality with (\ref{LE}), we obtain the following key inequality
\begin{alignat}{1}\label{fiolk22}
\sum_{i=1}^k\lambda_j(\Omega)\geq n\omega_n\frac{\big((2\pi)^{-n}V(\Omega) \big)^{2l+n+1}}{(2l+n)(2l+n+1)\rho^{2l+n}}\big((t+1)^{2l+n+1} -t^{2l+n+1} \big).
\end{alignat}

\section{In low dimensional cases}

In this section, we will establish lower bounds for higher eigenvalues in  low dimensional cases.

\subsection{When n=2}

In this section,  we have the following lower bound.

\begin{Theorem}\label{CoT1}
For any bounded domain $\Omega\subseteq \mathbb{R}^2$, let $\{\lambda_i\}$ be the eigenvalues  of the eigenvalue problem (\ref{EOE}).

If $l=2$, then
\begin{alignat}{1}\label{fiolkn2l2}
\sum_{i=1}^k\lambda_j(\Omega)\geq \frac{1}{3\omega_2^2}\left( \frac{(2\pi)^2}{V(\Omega)} \right)^2k^3+\frac{\pi}{3I(\Omega)}k^2-\frac{1}{145252}\frac{V^4(\Omega)}{I^3(\Omega)}.
\end{alignat}

If $l\geq 3$, and $k$ satisfies $k\geq \max\{k_1,k_2\}$, where $k_1$ and $k_2$ are defined by (\ref{fdok1}) and (\ref{fdok2}) respectively, then
\begin{alignat}{1}\label{eod21}
\begin{split}
\sum_{i=1}^k\lambda_j(\Omega)\geq & \frac{1}{(l+1)\omega^l_2}\alpha^{-l}k^{l+1}+\frac{l}{12\omega^{l-1}_2}\frac{(2\pi)^4}{VI}\alpha^{l3-l}k^l-c_{21}\frac{(2\pi)^8}{(VI)^2}\alpha^{6-l}k^{l-1}\\
&+c_{22}\frac{\pi^{12}}{(VI)^3}\alpha^{9-l}k^{l-2}-c_{23}\frac{\pi^{16}}{(VI)^4}\alpha^{12-l}k^{l-3},\\
\end{split}
\end{alignat}
where
\begin{alignat*}{1}
c_{21}=&\frac{(48l^2+4l-11)l}{2304(l+1)\omega^{l-2}_2},\,\,c_{22}= \frac{8l(3l+1)(l-1)}{649},\,\,\,c_{23}= \frac{(2l+1)l(l-1)(l-2)}{486\omega^{l-4}_2}.
\end{alignat*}

\end{Theorem}

\begin{Proofa}When $n=2$, we consider the following equation
\begin{alignat*}{1}
(t+1)^3-t^3=Q.
\end{alignat*}

The unique positive root if the above equation is $t=\sqrt{ {Q}/{3}-1/{12}}-1/{2}.$ For convenience, we define
\begin{alignat*}{1}
\tilde{t}= \frac{Q}{3}-\frac{1}{12},\,\,\,\,\tilde{t}_1=\frac{Q}{3}.
\end{alignat*}

Notice that $Q\geq 24$ when $n=2$. Hence, we have
\begin{alignat}{1}\label{qn2l2}
\tilde{t}^{\frac{1}{2}}\geq \frac{499}{500}\left( \frac{Q}{3} \right)^{\frac{1}{2}}.
\end{alignat}

\textbf{Case I.} If $l=2$, we have $(t+1)^{7}-t^{7}=7Q^3/27+7Q^2/9- {1}/{27}.$ Due to  (\ref{fiolk22}), we get (\ref{fiolkn2l2}).

\textbf{Case II.} If $l\geq3$, there holds
\begin{alignat}{1}\label{eqn2l21}
\begin{split}
{t}^{2l+2}=&\tilde{t}^{\frac{2l+2}{2}}-\frac{C^1_{2l+2}}{2}\tilde{t}^{\frac{2l+1}{2}}
+\frac{C^2_{2l+2}}{4}\tilde{t}^{\frac{2l}{2}}-\frac{C^3_{2l+2}}{8}\tilde{t}^{\frac{2l-1}{2}}+\cdots,\\
{t}^{2l+1}=&\tilde{t}^{\frac{2l+1}{2}}-\frac{C^1_{2l+1}}{2}\tilde{t}^{\frac{2l}{2}}
+\frac{C^2_{2l+1}}{4}\tilde{t}^{\frac{2l-1}{2}}-\frac{C^3_{2l+1}}{8}\tilde{t}^{\frac{2l-2}{2}}+\cdots
\end{split}
\end{alignat}
and
\begin{alignat}{1}\label{eqn2l22}
{t}^{2l}=&\tilde{t}^{\frac{2l}{2}}-\frac{C^1_{2l}}{2}\tilde{t}^{\frac{2l-1}{2}}+\cdots.
\end{alignat}

Using the inequality (\ref{qn2l2}), we also get
\begin{alignat}{1}\label{eqn2l23}
\begin{split}
\tilde{t}^{\frac{2l+2}{2}}=&\tilde{t}_1^{l+1}-\frac{C^1_{l+1}}{12}\tilde{t}_1^{l}+\frac{C^2_{l+1}}{144}\tilde{t}_1^{l-1}-\frac{C^3_{l+1}}{1728}\tilde{t}_1^{l-2}
+\cdots,\\
\tilde{t}^{\frac{2l}{2}}=&\tilde{t}_1^{l}-\frac{C^1_{l}}{12}\tilde{t}_1^{l-1}+\frac{C^2_{l}}{144}\tilde{t}_1^{l-2}-\frac{C^3_{l}}{1728}\tilde{t}_1^{l-3}+\cdots,\\
\tilde{t}=&\tilde{t}_1^{l-1}-\frac{C^1_{l-1}}{12}\tilde{t}_1^{l-2}+\cdots
\end{split}
\end{alignat}
and
\begin{alignat}{1}\label{eqn2l24}
\begin{split}
\tilde{t}^{\frac{2l+1}{2}}=&\tilde{t}^{\frac{1}{2}}\bigg(\tilde{t}_1^{l}-\frac{C^1_{l}}{12}\tilde{t}_1^{l-1}+\cdots\bigg)\geq \frac{499}{500}\bigg(\tilde{t}_1^{l+\frac{1}{2}}-\frac{C^1_{l}}{12}\tilde{t}_1^{l-\frac{1}{2}}+\cdots\bigg),\\
\tilde{t}^{\frac{2l-1}{2}}=&\tilde{t}^{\frac{1}{2}}\bigg(\tilde{t}_1^{l-1}-\frac{C^1_{l-1}}{12}\tilde{t}_1^{l-2}+\cdots\bigg)\geq \frac{499}{500}\bigg(\tilde{t}_1^{l-\frac{1}{2}}-\frac{C^1_{l-1}}{12}\tilde{t}_1^{l-\frac{3}{2}}+\cdots\bigg).
\end{split}
\end{alignat}

From (\ref{eqn2l21})-(\ref{eqn2l24}), we obtain
\begin{alignat}{1}\label{n2l211}
\begin{split}
(t+1)^{2l+3}-t^{2l+3}=&C^1_{2l+3}\tilde{t}^{\frac{2l+2}{2}}-\frac{C^1_{2l+3}C^1_{2l+2}}{2}\tilde{t}^{\frac{2l+1}{2}}
+C^2_{2l+3}\tilde{t}^{\frac{2l+1}{2}}\\
&+\frac{C^1_{2l+3}C^2_{2l+2}}{4}\tilde{t}^{\frac{2l}{2}}-\frac{C^2_{2l+3}C^1_{2l+1}}{2}\tilde{t}^{\frac{2l}{2}}+C^3_{2l+3}\tilde{t}^{\frac{2l}{2}}\\
&-\frac{C^1_{2l+3}C^3_{2l+2}}{8}\tilde{t}^{\frac{2l-1}{2}}+\frac{C^2_{2l+3}C^2_{2l+1}}{4}\tilde{t}^{\frac{2l-1}{2}}\\
&-\frac{C^3_{2l+3}C^1_{2l}}{2}\tilde{t}^{\frac{2l-1}{2}}+C^4_{2l+3}\tilde{t}^{\frac{2l-1}{2}}
-\frac{C^2_{2l+3}C^3_{2l+1}}{8}\tilde{t}^{\frac{2l-2}{2}}+\cdots.
\end{split}
\end{alignat}

Obviously, if we choose
\begin{alignat}{1}\label{fdok1}
k\geq k_1:=\frac{3}{2}+\frac{1}{24}\max_{1\leq i\leq 2l-1} \bigg\{\left( \frac{C^{i+1}_{2l+2}}{2C^i_{2l+2}} \right)^2 ,\left( \frac{C^{i+1}_{2l+1}}{2C^i_{2l+1}}\right)^2,\left( \frac{C^{i+1}_{2l}}{2C^i_{2l}} \right)^2\bigg\},
\end{alignat}
then, we obtain
\begin{alignat}{1}\label{eqn2l25}
\begin{split}
{t}^{2l+2}\geq&\tilde{t}^{\frac{2l+2}{2}}-\frac{C^1_{2l+2}}{2}\tilde{t}^{\frac{2l+1}{2}}
+\frac{C^2_{2l+2}}{4}\tilde{t}^{\frac{2l}{2}}-\frac{C^3_{2l+2}}{8}\tilde{t}^{\frac{2l-1}{2}},\\
{t}^{2l+1}\geq&\tilde{t}^{\frac{2l+1}{2}}-\frac{C^1_{2l+1}}{2}\tilde{t}^{\frac{2l}{2}}
+\frac{C^2_{2l+1}}{4}\tilde{t}^{\frac{2l-1}{2}}-\frac{C^3_{2l+1}}{8}\tilde{t}^{\frac{2l-2}{2}},\\
{t}^{2l}\geq&\tilde{t}^{\frac{2l}{2}}-\frac{C^1_{2l}}{2}\tilde{t}^{\frac{2l-1}{2}}.
\end{split}
\end{alignat}
Combing with (\ref{n2l211}) and (\ref{eqn2l25}), we infer that
\begin{alignat}{1}\label{eqn2l26}
\begin{split}
(t+1)^{2l+3}-t^{2l+3}\geq&C^1_{2l+3}\tilde{t}^{\frac{2l+2}{2}}
+\frac{(2l+3)(2l+2)(2l+1)}{24}\tilde{t}^{\frac{2l}{2}}\\
&-\frac{(2l+3)(2l+2)(2l+1)2l(2l-1)}{96}\tilde{t}^{\frac{2l-2}{2}}.
\end{split}
\end{alignat}

If we choose $k\geq \max\{k_1,k_2\}$, where
\begin{alignat}{1}\label{fdok2}
k_2=\max_{1\leq i\leq l-1} \bigg\{\frac{C^{i+1}_{l+1}}{96C^i_{l+1}},\frac{C^{i+1}_{l-1}}{96C^i_{l-1}}\bigg\},
\end{alignat}
then there holds
\begin{alignat}{1}\label{eqn2l27}
\begin{split}
\tilde{t}^{\frac{2l+2}{2}}\geq&\tilde{t}_1^{l+1}-\frac{C^1_{l+1}}{12}\tilde{t}_1^{l}+\frac{C^2_{l+1}}{144}\tilde{t}_1^{l-1}-\frac{C^3_{l+1}}{1728}\tilde{t}_1^{l-2},\\
\tilde{t}^{\frac{2l}{2}}\geq&\tilde{t}_1^{l}-\frac{C^1_{l}}{12}\tilde{t}_1^{l-1}+\frac{C^2_{l}}{144}\tilde{t}_1^{l-2}-\frac{C^3_{l}}{1728}\tilde{t}_1^{l-3}\\
\end{split}
\end{alignat}
and
\begin{alignat}{1}\label{eqn2l28}
\begin{split}
\tilde{t}^{\frac{2l-2}{2}}\leq&\tilde{t}_1^{l-1},
\end{split}
\end{alignat}
where we have used the fact that $Q\geq 24 k$ to obtain $k_2$.

Collecting (\ref{eqn2l26})-(\ref{eqn2l28}) together, we arrive at
\begin{alignat*}{1}
(t+1)^{2l+3}-t^{2l+3}\geq& (2l+3)\tilde{t}_1^{l+1}+\frac{(2l+3)(2l+2)l}{6}\tilde{t}_1^{l}\\
&-\frac{(48l^2+4l-11)(2l+3)l}{288}\tilde{t}_1^{l-1}\\
&+\frac{(3l+1)(2l+3)(l+1)l(l-1)}{5284}\tilde{t}_1^{l-2}\\
&-\frac{(2l+3)(2l+2)(2l+1)l(l-1)(l-2)}{248832}\tilde{t}_1^{l-3}.
\end{alignat*}

Therefore, combing with (\ref{fiolk22}), we  prove the eigenvalue inequality (\ref{eod21}).
\cvd
\end{Proofa}

\subsection{When n=3}

By using  (\ref{fiolk22}), we have the following lower bound.

\begin{Theorem}\label{CoT2}
For any bounded domain $\Omega\subseteq \mathbb{R}^3$, let $\{\lambda_i\}$ be the eigenvalues  of the eigenvalue problem (\ref{EOE}).

 If $l=1$, then
\begin{alignat}{1}\label{3c1}
\begin{split}
\sum_{i=1}^k\lambda_j(\Omega)\geq& \frac{3}{5}\left(\frac{(2\pi)^3}{\omega_3 V } \right)^{\frac{2}{3}}k^{\frac{5}{3}}+\frac{1}{16}\frac{V}{I}k
-\frac{11}{3840}\left(\frac{\omega_3^{\frac{1}{3}}}{2\pi}\right)^2\left(\frac{V^{\frac{4}{3}} }{I }\right)^2k^{\frac{1}{3}}\\
&+\frac{659}{6400000}\left( \frac{\omega_3^4}{4} \right)^{\frac{1}{3}}\frac{V^{\frac{13}{3}} }{(2\pi)^4I^3 }k^{-\frac{1}{3}}.
\end{split}
\end{alignat}

If $l=2$, then
\begin{alignat}{1}\label{3c2}
\begin{split}
\sum_{i=1}^k\lambda_j(\Omega)\geq& \frac{3}{7}\left(\frac{(2\pi)^3}{\omega_3 } \right)^{\frac{4}{3}}k^{\frac{7}{3}}+\frac{\pi^2}{2\omega_3^{\frac{2}{3}}}\frac{V^{\frac{1}{3}} }{I }k^{\frac{5}{3}}
-\frac{1}{256}\left(\frac{V }{I }\right)^2k\\
&+\frac{23}{580608}\frac{V^{\frac{1}{3}} }{(2\pi)^2 I^3 }k^{\frac{1}{3}}.
\end{split}
\end{alignat}

If $l\geq 3$ and $k\geq \max\{k_1,k_2\}$, where $k_1,k_2$ are defined by (\ref{n3k1}) and (\ref{n3k2}),  then
\begin{alignat}{1}\label{3c3}
\begin{split}
\sum_{i=1}^k\lambda_j(\Omega)\geq& \frac{3}{(2l+3)\omega^{\frac{2}{3}}_3}\alpha^{-\frac{2}{3}}k^{\frac{2l+3}{3}}
+\frac{4l\pi^6}{VI\omega^{\frac{2l-2}{3}}_3}\alpha^{\frac{8-2l}{3}}k^{\frac{2l+1}{3}}\\
&-\frac{(l+1)^2}{2\omega^{\frac{2l-3}{3}}_3\rho^3}\alpha^{\frac{4-2l}{3}}k^{\frac{2l}{3}}
+\frac{(3l-2)(2l+1)(l+1)}{24\omega^{\frac{2l-4}{3}}\rho^4}\alpha^{\frac{8-2l}{3}}k^{\frac{2l-1}{3}}\\
&-\frac{(2l+1)(2l-1)(l+1)}{24\omega^{\frac{2l-5}{3}}\rho^5}\alpha^{\frac{12-2l}{3}}k^{\frac{2l-2}{3}}.
\end{split}
\end{alignat}

\end{Theorem}

\begin{Proofb}
When $n=3$, we consider the following equation
\begin{alignat*}{1}
(t+1)^4-t^4=Q.
\end{alignat*}
Then, the positive root of above  equation  is
\begin{alignat*}{1}
t=\frac{1}{2}\Bigg(\left(Q+\sqrt{Q^2+\frac{1}{27}} \right)^{\frac{1}{3}} -\left(-Q+\sqrt{Q^2+\frac{1}{27}} \right)^{\frac{1}{3}}  -1\Bigg).
\end{alignat*}
Using the fact that $Q\geq 1$, we get
\begin{alignat}{1}\label{n3eot}
\begin{split}
t\geq& \hat{t}:= -\frac{1}{2}+\frac{1}{4^{\frac{1}{3}}}Q^{\frac{1}{3}}-\frac{1}{6}\frac{1}{2^{\frac{1}{3}}}Q^{-\frac{1}{3}}
+\frac{1}{324}\frac{1}{4^{\frac{1}{3}}}Q^{-\frac{5}{3}}.
\end{split}
\end{alignat}

For convenience, we define
\begin{alignat*}{1}
\tilde{t}_1= -\frac{1}{2}+\frac{1}{4^{\frac{1}{3}}}Q^{\frac{1}{3}}.
\end{alignat*}

\textbf{Case I.} If $l=1$, then
\begin{alignat*}{1}
(\hat{t}+1)^{6} -\hat{t}^{6}\geq \frac{3}{4}\frac{1}{2^{\frac{1}{3}}}Q^{\frac{5}{3}}+\frac{5}{8}Q-\frac{11}{24}\frac{1}{2^{\frac{2}{3}}}Q^{\frac{1}{3}}+\frac{659}{10000}Q^{-\frac{1}{3}}.
\end{alignat*}
Collecting the above estimates  with (\ref{fiolk22}), we get (\ref{3c1}).

\textbf{Case II.} If $l=2$, then
\begin{alignat*}{1}
(\hat{t}+1)^{8} -\hat{t}^{8}\geq \frac{1}{2}\frac{1}{2^{\frac{2}{3}}}Q^{\frac{7}{3}}+\frac{7}{6}\frac{1}{2^{\frac{1}{3}}}Q^{\frac{5}{3}}-\frac{7}{24}Q
+\frac{23}{324}\frac{1}{2^{\frac{2}{3}}}Q^{\frac{1}{3}}.
\end{alignat*}
Then, we obtain (\ref{3c2}).

\textbf{Case III.} If $l\geq 3$, for convenience, we define
\begin{alignat*}{1}
\bar{t}:=-\frac{1}{2}+\frac{1}{4^{\frac{1}{3}}}Q^{\frac{1}{3}}-\frac{1}{6}\frac{1}{2^{\frac{1}{3}}}Q^{-\frac{1}{3}}.
\end{alignat*}
Hence, we get
\begin{alignat*}{1}
\bar{t}^{2l+3}=&\tilde{t}_1^{2l+3}
-\frac{C^1_{2l+3}}{6}\frac{1}{2^{\frac{1}{3}}}Q^{-\frac{1}{3}}\tilde{t}_1^{2l+2}+\cdots,\\
\bar{t}^{2l+2}=&\tilde{t}_1^{2l+2}
-\frac{C^1_{2l+2}}{6}\frac{1}{2^{\frac{1}{3}}}Q^{-\frac{1}{3}}\tilde{t}_1^{2l+1}+\cdots,\\
\bar{t}^{2l+1}=&\tilde{t}_1^{2l+1}
-\frac{C^1_{2l+1}}{6}\frac{1}{2^{\frac{1}{3}}}Q^{-\frac{1}{3}}\tilde{t}_1^{2l}+\cdots.
\end{alignat*}
Moreover, we also have
\begin{alignat*}{1}
\tilde{t}_1^{2l+3}=&\left(\frac{Q}{4} \right)^{\frac{2l+3}{3}}-\frac{C^1_{2l+3}}{2}\left(\frac{Q}{4} \right)^{\frac{2l+2}{3}}+\frac{C^2_{2l+3}}{4}\left(\frac{Q}{4} \right)^{\frac{2l+1}{3}}-\frac{C^3_{2l+3}}{8}\left(\frac{Q}{4} \right)^{\frac{2l}{3}}+\cdots,\\
\tilde{t}_1^{2l+2}=&\left(\frac{Q}{4} \right)^{\frac{2l+2}{3}}-\frac{C^1_{2l+2}}{2}\left(\frac{Q}{4} \right)^{\frac{2l+1}{3}}+\frac{C^2_{2l+2}}{4}\left(\frac{Q}{4} \right)^{\frac{2l}{3}}-\frac{C^3_{2l+2}}{8}\left(\frac{Q}{4} \right)^{\frac{2l-1}{3}}+\cdots,\\
\tilde{t}_1^{2l+1}=&\left(\frac{Q}{4} \right)^{\frac{2l+1}{3}}-\frac{C^1_{2l+1}}{2}\left(\frac{Q}{4} \right)^{\frac{2l}{3}}+\frac{C^2_{2l+1}}{4}\left(\frac{Q}{4} \right)^{\frac{2l-1}{3}}-\frac{C^3_{2l+1}}{8}\left(\frac{Q}{4} \right)^{\frac{2l-2}{3}}+\cdots.
\end{alignat*}

Consequently, we arrive at
\begin{alignat*}{1}
(\bar{t}+1)^{2l+4} -\bar{t}^{2l+4}=&C^1_{2l+4}\tilde{t}_1^{2l+3}
-\frac{C^1_{2l+4}C^1_{2l+3}}{6}\frac{1}{2^{\frac{1}{3}}}Q^{-\frac{1}{3}}\tilde{t}_1^{2l+2}\\
&+C^2_{2l+4}\tilde{t}_1^{2l+2}-\frac{C^2_{2l+4}C^1_{2l+2}}{6}\frac{1}{2^{\frac{1}{3}}}Q^{-\frac{1}{3}}\tilde{t}_1^{2l+1}\\
&+C^3_{2l+4}\tilde{t}_1^{2l+1}-\frac{C^3_{2l+4}C^1_{2l+1}}{6}\frac{1}{2^{\frac{1}{3}}}Q^{-\frac{1}{3}}\tilde{t}_1^{2l}+\cdots.
\end{alignat*}

If  $k$ satisfy $k\geq k_1$, where
\begin{alignat}{1}\label{n3k1}
k_1:=\frac{2}{105}\max_{1\leq i\leq 2l}\bigg\{ \left(  \frac{C^{i+1}_{2l+3}}{44C^i_{2l+3}}+\frac{1}{3}\right)^3, \left(  \frac{C^{i+1}_{2l+2}}{44C^i_{2l+2}}+\frac{1}{3}\right)^3,
 \left(  \frac{C^{i+1}_{2l+1}}{44C^i_{2l+1}}+\frac{1}{3}\right)^3 \bigg\},
\end{alignat}
then, it follows that
\begin{alignat*}{1}
\bar{t}^{2l+3}\geq&\left(-\frac{1}{2}+\frac{1}{4^{\frac{1}{3}}}Q^{\frac{1}{3}}\right)^{2l+3}
-\frac{C^1_{2l+3}}{6}\frac{1}{2^{\frac{1}{3}}}Q^{-\frac{1}{3}}\left(-\frac{1}{2}+\frac{1}{4^{\frac{1}{3}}}Q^{\frac{1}{3}}\right)^{2l+2},\\
\bar{t}^{2l+2}\geq&\left(-\frac{1}{2}+\frac{1}{4^{\frac{1}{3}}}Q^{\frac{1}{3}}\right)^{2l+2}
-\frac{C^1_{2l+2}}{6}\frac{1}{2^{\frac{1}{3}}}Q^{-\frac{1}{3}}\left(-\frac{1}{2}+\frac{1}{4^{\frac{1}{3}}}Q^{\frac{1}{3}}\right)^{2l+1},\\
\bar{t}^{2l+1}\geq&\left(-\frac{1}{2}+\frac{1}{4^{\frac{1}{3}}}Q^{\frac{1}{3}}\right)^{2l+1}
-\frac{C^1_{2l+1}}{6}\frac{1}{2^{\frac{1}{3}}}Q^{-\frac{1}{3}}\left(-\frac{1}{2}+\frac{1}{4^{\frac{1}{3}}}Q^{\frac{1}{3}}\right)^{2l},
\end{alignat*}
where we have used the fact that $Q> 210$ when $n=3$.

According to the above equations, one can immediately get
\begin{alignat}{1}\label{n3lg31}
\begin{split}
(\bar{t}+1)^{2l+4} -\bar{t}^{2l+4}\geq&C^1_{2l+4}\tilde{t}_1^{2l+3}
-\frac{C^1_{2l+4}C^1_{2l+3}}{6}\frac{1}{2^{\frac{1}{3}}}Q^{-\frac{1}{3}}\tilde{t}_1^{2l+2}\\
&+C^2_{2l+4}\tilde{t}_1^{2l+2}
-\frac{C^2_{2l+4}C^1_{2l+2}}{6}\frac{1}{2^{\frac{1}{3}}}Q^{-\frac{1}{3}}\tilde{t}_1^{2l+1}\\
&+C^3_{2l+4}\tilde{t}_1^{2l+1}
-\frac{C^3_{2l+4}C^1_{2l+1}}{6}\frac{1}{2^{\frac{1}{3}}}Q^{-\frac{1}{3}}\tilde{t}_1^{2l}.
\end{split}
\end{alignat}

Furthermore, since $l\geq 3$ and $Q> 210$, if we assume that
\begin{alignat}{1}\label{n3k2}
k\geq k_2:=\frac{2}{105}\max_{ 1\leq i\leq 2l } \bigg\{  \left( \frac{C^{i+1}_{2l+3}}{2C^i_{2l+3}}\right)^3, \left( \frac{C^{i+1}_{2l+2}}{2C^i_{2l+2}}\right)^3,
 \left( \frac{C^{i+1}_{2l+1}}{2C^i_{2l+1}}\right)^3\bigg\},
\end{alignat}
we can also get
\begin{alignat}{1}\label{n3lg32}
\begin{split}
\tilde{t}_1^{2l+3}\geq&\left(\frac{Q}{4} \right)^{\frac{2l+3}{3}}-\frac{C^1_{2l+3}}{2}\left(\frac{Q}{4} \right)^{\frac{2l+2}{3}}
+\frac{C^2_{2l+3}}{4}\left(\frac{Q}{4} \right)^{\frac{2l+1}{3}}-\frac{C^3_{2l+3}}{8}\left(\frac{Q}{4} \right)^{\frac{2l}{3}},\\
\tilde{t}_1^{2l+2}\geq&\left(\frac{Q}{4} \right)^{\frac{2l+2}{3}}-\frac{C^1_{2l+2}}{2}\left(\frac{Q}{4} \right)^{\frac{2l+1}{3}}
+\frac{C^2_{2l+2}}{4}\left(\frac{Q}{4} \right)^{\frac{2l}{3}}-\frac{C^3_{2l+2}}{8}\left(\frac{Q}{4} \right)^{\frac{2l-1}{3}},\\
\tilde{t}_1^{2l+1}\geq&\left(\frac{Q}{4} \right)^{\frac{2l+1}{3}}-\frac{C^1_{2l+1}}{2}\left(\frac{Q}{4} \right)^{\frac{2l}{3}}
+\frac{C^2_{2l+1}}{4}\left(\frac{Q}{4} \right)^{\frac{2l-1}{3}}-\frac{C^3_{2l+1}}{8}\left(\frac{Q}{4} \right)^{\frac{2l-2}{3}}.
\end{split}
\end{alignat}

Combining with (\ref{n3lg31}) and (\ref{n3lg32}), we can prove the following inequality
\begin{alignat*}{1}
\begin{split}
(\bar{t}+1)^{2l+4} -\bar{t}^{2l+4}\geq&(2l+4)\left(\frac{Q}{4} \right)^{\frac{2l+3}{3}}\\
&+\frac{(2l+3)(l+2)l}{6}\left(\frac{Q}{4} \right)^{\frac{2l+1}{3}}\\
&-\frac{(2l+3)(l+2)(l+1)^2}{3}\left(\frac{Q}{4} \right)^{\frac{2l}{3}}\\
&+\frac{(3l-2)(2l+3)(2l+1)(l+2)(l+1)}{36}\left(\frac{Q}{4} \right)^{\frac{2l-1}{3}}\\
&-\frac{(2l+3)(2l+1)(2l-1)(l+2)(l+1)}{36}\left(\frac{Q}{4} \right)^{\frac{2l-2}{3}}.
\end{split}
\end{alignat*}

Finally, combing with (\ref{fiolk22}), we get the eigenvalue inequality (\ref{3c3}).
\cvd
\end{Proofb}

\subsection{When n=4} 

When $n=4$, we have the following lower bound.

\begin{Theorem}\label{CoT3}
For any bounded domain $\Omega\subseteq \mathbb{R}^4$, let $\{\lambda_i\}$ be the eigenvalues  of the eigenvalue problem (\ref{EOE}).

If $l=1$, then
\begin{alignat}{1}\label{4c1}
\begin{split}
\sum_{i=1}^k\lambda_j(\Omega)\geq& \frac{2}{3\omega^{\frac{1}{2}_4}}\left( \frac{(2\pi)^4}{V } \right)^{\frac{1}{2}}k^{\frac{3}{2}}+\frac{1}{12}\frac{V }{I }k\\
&-\frac{7\omega^{\frac{1}{2}}_4}{960(2\pi)^2}\frac{V^2 }{I^2 }k^{\frac{1}{2}}+\frac{47\omega_4}{114688(2\pi)^4}\frac{V^4 }{I^3 }\\
&+\frac{\omega^{\frac{5}{4}}}{150\rho^7}\left( \frac{V }{(2\pi)^4} \right)^{\frac{33}{4}}k^{-\frac{1}{4}}-\frac{449}{122880}\frac{\omega^{\frac{3}{2}}_4}{\rho^8}\left( \frac{V }{(2\pi)^4} \right)^{\frac{19}{2}}k^{-\frac{1}{2}}.
\end{split}
\end{alignat}

If $l=2$, then
\begin{alignat}{1} \label{4c2}
\begin{split}
\sum_{i=1}^k\lambda_j(\Omega)\geq& \frac{1}{2\omega_4}\frac{(2\pi)^4}{V }k^2+\frac{\pi^2}{12\omega^{\frac{1}{2}}_4}\frac{V^{\frac{1}{2}} }{I }k^{\frac{3}{2}}
-\frac{1}{80}\frac{V^2 }{I^2 }k\\
&+\frac{61\omega^{\frac{1}{2}}_4}{1280\rho^6}\left(  \frac{V }{(2\pi)^4}\right)^\frac{13}{2}k^{\frac{1}{2}}
+\frac{\omega^{\frac{3}{4}}}{150\rho^7}\left(  \frac{V }{(2\pi)^4}\right)^{\frac{31}{4}}k^{\frac{1}{4}}\\
&-\frac{17423\omega_4}{1843200\rho^8}\left(  \frac{V }{(2\pi)^4}\right)^{9}+\frac{7\omega^{\frac{5}{4}}_4}{1200\rho^9}\left(  \frac{V }{(2\pi)^4}\right)^{\frac{41}{4}}k^{-\frac{1}{4}}.
\end{split}
\end{alignat}

If $l\geq 3$ and $k\geq \max\{k_1,k_2\}$, where $k_1,k_2$ are defined by (\ref{n4k1}) and (\ref{n4k2}),  then
\begin{alignat}{1}\label{4c3}
\begin{split}
\sum_{i=1}^k\lambda_j(\Omega)\geq& \frac{2}{(l+2)\omega_4^{\frac{l}{2}}}\alpha^{-\frac{l}{2}}k^{\frac{l+2}{2}}
+\frac{(2\pi)^8l}{12\omega_4^{\frac{l-1}{2}}V I }\alpha^{\frac{5-l}{2}}k^{\frac{l+1}{2}}\\
&-\frac{128(4l+7)(2l+1)\pi^{12}}{3\omega_4^{\frac{2l-3}{4}}(VI)^{\frac{3}{2}}}\alpha^{\frac{15-2l}{4}}k^{\frac{2l+1}{4}}\\
&+\frac{(2l+3)(2l-1)(l+1)}{12\omega_4^{\frac{l-2}{2}}\rho^4}\alpha^{\frac{10-l}{2}}k^{\frac{l}{2}}\\
&-\frac{(2l+3)(2l+1)(l+1)l}{18\omega_4^{\frac{2l-5}{4}}\rho^5}\alpha^{\frac{25-2l}{4}}k^{\frac{2l-1}{4}}.
\end{split}
\end{alignat}

\end{Theorem}

\begin{Proofc}When $n=4$, we consider the following critical equation
\begin{alignat}{1}\label{efo}
\left(t+\frac{1}{2}\right)^{5}-\left(t-\frac{1}{2}\right)^{5}=Q
\end{alignat}
which is equals to $5t^4+5t^2/2+\ {5}/{80}=Q.$ Therefore, the positive root of (\ref{efo}) satisfies $t\leq t_1:=(  Q/5)^{ 1/4}$. Let $c_1:=t_1-t$, we have
\begin{alignat*}{1}
5(t_1-c_1)^4+\frac{5}{2}(t_1-c_1)^2+\frac{5}{80}=Q.
\end{alignat*}
By direct calculation, we get
\begin{alignat}{1}\label{fc1}
\left(\frac{5}{2}c_1^2-20c_1^3t_1\right)+\left(30c_1^2t_1^2-5c_1t_1+\frac{5}{80}+5c_1^4\right)+\left(\frac{5}{2}t_1^2-20c_1t^3_1\right)=0.
\end{alignat}

Firstly, we deal with the first bracket on the left side of (\ref{fc1}).  Since $t>0$, we infer that
\begin{alignat}{1}\label{doc11}
-20c_1^3t_1+\frac{5}{2}c_1^2\leq -\frac{35}{2}c_1^3t_1.
\end{alignat}

Secondly, we will consider the second bracket.  Without loss of generality, we further restrict $c_1$ as $c_1\leq ( t_1/10)^{1/3}$
and
\begin{alignat}{1}\label{doc12}
\frac{9+\sqrt{51}}{120} \frac{1}{t_1} \leq c_1\leq \frac{9+\sqrt{51}}{120} \frac{1}{t_1},
\end{alignat}
then, we obtain $-0.5c_1t_1+5c_1^4\leq 0, 30c_1^2t_1^2-9c_1t_1/2+{5}/{80}\leq 0,$ which imply
\begin{alignat*}{1}
\left(30c_1^2t_1^2-5c_1t_1+\frac{5}{80}+5c_1^4\right)\leq 0.
\end{alignat*}

Notice that
\begin{alignat}{1}\label{doc13}
\left( \frac{t_1}{10} \right)^{\frac{1}{3}}\geq 0.464 t_1^{\frac{1}{3}}
\end{alignat}
and
\begin{alignat}{1}\label{doc14}
\begin{split}
\frac{9-\sqrt{51}}{120}   \approx 0.015488,\,\,\,\frac{9+\sqrt{51}}{120}  \approx 0.134512.
\end{split}
\end{alignat}

Thirdly, we with deal with the third bracket. By choosing $c_1\geq 1/(8t_1)$,
we conclude that $5t_1^2/2-20c_1t^3_1 \leq0.$ According to the similar discussion in the next section, we can also prove that $t_1>1$. Hence
\begin{alignat}{1}\label{doc15}
\frac{1}{8t_1}\leq 0.464 t_1^{\frac{1}{3}}.
\end{alignat}
Due to (\ref{doc11})-(\ref{doc15}), we conclude that $t=t_1-c_1\geq t_1-1/(8 t_1).$ Hence, we assume
\begin{alignat}{1}\label{fc2}
t\geq t_1-\frac{1}{8 t_1}+c_2,
\end{alignat}
where $c_2$ is a positive number to be determined. Putting (\ref{fc2}) into (\ref{efo}), we get
\begin{alignat}{1}\label{eof2}
\begin{split}
 0=&20c_2t_1^3+30c_2^2t_1^2+\frac{5}{2}\left( 8c_2^3-c_2 \right)t_1+\frac{1}{32}\left(160c_2^4-160c_2^2-3  \right)\\
 &-\frac{15}{32}\frac{c_2^2}{t_1^2}-\frac{5}{16}\left( 8c_2^2-c_2\right)\frac{1}{t_1}-\frac{5}{128}\frac{c_2}{t_1^3}-\frac{5}{4096}\frac{1}{t_1^4}.
\end{split}
\end{alignat}
If we choose $c_2=\frac{3}{640}\frac{1}{t_1^3}$, then the right side of (\ref{eof2}) becomes
\begin{alignat*}{1}
-\frac{3}{256}\frac{1}{t_1^2}+\frac{137}{40960}\frac{1}{t_1^4}-\frac{3}{10240}\frac{1}{t_1^6}+q(t_1),
\end{alignat*}
where
\begin{alignat*}{1}
q(t_1)=O(t_1^{-8})\leq \frac{81}{6553600}\frac{1}{t_1^8}
\end{alignat*}
provide $t_1\geq 1$. Hence, we conclude that $c_2>\frac{3}{640}\frac{1}{t_1^3}$. Therefore, we can further conclude that
\begin{alignat}{1}\label{fc3}
t\geq t_1-\frac{1}{8 t_1}+\frac{3}{640}\frac{1}{t_1^3}+c_3,
\end{alignat}
where $c_3$ is a positive number. Let $c_3={1}/(600t_1^6)$ and put $t^{*}:=t_1-1/({8 t_1})+{3}/(640 t_1^3)+c_3$
into the critical equation, we arrive at
\begin{alignat}{1}\label{eoq2}
\begin{split}
Q_1=5t_1^4-\frac{3}{256}\frac{1}{t_1^2}+\frac{2}{65}\frac{1}{t_1^4}+\frac{137}{40960}\frac{1}{t_1^4}+q_2(t),
\end{split}
\end{alignat}
where
\begin{alignat*}{1}
q_2(t)=O(t_1^{-5})\leq -\frac{1}{260}\frac{1}{t_1^5},
\end{alignat*}
provide $t_1\geq 2.5$. Combing with
\begin{alignat*}{1}
\begin{split}
\frac{3}{256}\frac{1}{t_1^2}\geq \frac{2}{65}\frac{1}{t_1^4}
\end{split}
\end{alignat*}
for $t_1\geq 2.5$, we conclude that
\begin{alignat*}{1}
t\geq t_1-\frac{1}{8 t_1}+\frac{3}{640}\frac{1}{t_1^3}+\frac{1}{600}\frac{1}{t_1^6}.
\end{alignat*}

Let $\eta=t-\frac{1}{2}$, then $\eta$ satisfies
\begin{alignat*}{1}
(\eta+1)^5-\eta^5=Q
\end{alignat*}
and
\begin{alignat}{1}\label{4t1}
\eta\geq -\frac{1}{2}+t_1-\frac{1}{8 t_1}+\frac{3}{640}\frac{1}{t_1^3}+\frac{1}{600}\frac{1}{t_1^6}.
\end{alignat}

For convenience, we define
\begin{alignat*}{1}
t_2=t_1-\frac{1}{2}= \frac{1}{5^{\frac{1}{4}}}Q^{\frac{1}{4}}-\frac{1}{2}.
\end{alignat*}

\textbf{Case I.} If $l=1$, we have
\begin{alignat*}{1}
(\eta+1)^7-\eta^7\geq& 7\left( \frac{Q}{5}\right)^{\frac{3}{2}}+\frac{7}{2}\left( \frac{Q}{5}\right)-\frac{49}{40}\left( \frac{Q}{5}\right)^{\frac{1}{2}}+\frac{141}{512}\\
&+\frac{7}{100}\left( \frac{Q}{5}\right)^{-\frac{1}{4}}-\frac{3143}{81920}\left( \frac{Q}{5}\right)^{-\frac{1}{2}}.
\end{alignat*}
By using the above estimate and (\ref{fiolk22}), we get (\ref{4c1}).

\textbf{Case II.} If $l=2$, we have
\begin{alignat*}{1}
(\eta+1)^9-\eta^9\geq& 9\left( \frac{Q}{5}\right)^{2}+12\left( \frac{Q}{5}\right)^{\frac{3}{2}}-\frac{18}{5}\left( \frac{Q}{5}\right)+\frac{549}{640}\left( \frac{Q}{5}\right)^{\frac{1}{2}}\\
&+\frac{3}{25}\left( \frac{Q}{5}\right)^{\frac{1}{4}}-\frac{17423}{102400}+\frac{21}{200}\left( \frac{Q}{5}\right)^{-\frac{1}{4}}.
\end{alignat*}
Using again (\ref{fiolk22}), we get (\ref{4c2})

\textbf{Case III.} If $l\geq3$, for convenience, we define
\begin{alignat*}{1}
\theta=-\frac{1}{2}+ t_1-\frac{1}{8 t_1}.
\end{alignat*}
Then we have
\begin{alignat*}{1}
(\eta+1)^{2l+5}-\eta^{2l+5}\geq (\theta+1)^{2l+5}-\theta^{2l+5}
\end{alignat*}
By direct calculation, we have
\begin{alignat*}{1}
\theta^{2l+4}=\left(t_1-\frac{1}{2}  \right)^{2l+4}-\frac{C^1_{2l+4}}{8t_1}\left(t_1-\frac{1}{2}  \right)^{2l+3}+\cdots,\\
\theta^{2l+3}=\left(t_1-\frac{1}{2}  \right)^{2l+3}-\frac{C^1_{2l+3}}{8t_1}\left(t_1-\frac{1}{2}  \right)^{2l+2}+\cdots,\\
\theta^{2l+2}=\left(t_1-\frac{1}{2}  \right)^{2l+2}-\frac{C^1_{2l+2}}{8t_1}\left(t_1-\frac{1}{2}  \right)^{2l+1}+\cdots.
\end{alignat*}
Furthermore, we also have the following equations
\begin{alignat*}{1}
t_2^{2l+4}=&\left( \frac{Q}{5} \right)^{\frac{2l+4}{4}}-\frac{C^1_{2l+4}}{2}\left( \frac{Q}{5} \right)^{\frac{2l+3}{4}}+\frac{C^2_{2l+4}}{4}\left( \frac{Q}{5} \right)^{\frac{2l+2}{4}}-\frac{C^3_{2l+4}}{8}\left( \frac{Q}{5} \right)^{\frac{2l+1}{4}}+\cdots,\\
t_2^{2l+3}=&\left( \frac{Q}{5} \right)^{\frac{2l+3}{4}}-\frac{C^1_{2l+3}}{2}\left( \frac{Q}{5} \right)^{\frac{2l+2}{4}}+\frac{C^2_{2l+3}}{4}\left( \frac{Q}{5} \right)^{\frac{2l+1}{4}}-\frac{C^3_{2l+3}}{8}\left( \frac{Q}{5} \right)^{\frac{2l}{4}}+\cdots,\\
t_2^{2l+2}=&\left( \frac{Q}{5} \right)^{\frac{2l+2}{4}}-\frac{C^1_{2l+2}}{2}\left( \frac{Q}{5} \right)^{\frac{2l+1}{4}}+\frac{C^2_{2l+1}}{4}\left( \frac{Q}{5} \right)^{\frac{2l}{4}}-\frac{C^3_{2l+2}}{8}\left( \frac{Q}{5} \right)^{\frac{2l-1}{4}}+\cdots.
\end{alignat*}

Hence, we can derive that
\begin{alignat*}{1}
(\theta+1)^{2l+5}-\theta^{2l+5}=&C^1_{2l+5}t_2^{2l+4}-\frac{C^1_{2l+5}C^1_{2l+4}}{8t_1}t_2^{2l+3}
+C^2_{2l+5}t_2^{2l+3}\\
&-\frac{C^2_{2l+5}C^1_{2l+3}}{8t_1}t_2^{2l+2}+C^3_{2l+5}t_2^{2l+2}-\frac{C^3_{2l+5}C^1_{2l+2}}{8t_1}t_2^{2l+1}+\cdots.
\end{alignat*}

If $k$ satisfies $k\geq k_1$, where
\begin{alignat}{1}\label{n4k1}
 k_1:=\frac{1}{455} \max_{1\leq i\leq 2l+3} \bigg\{ \left(\frac{C^{i+1}_{2l+4}}{27C^{i}_{2l+4}} +\frac{1}{2}\right)^4,\left(\frac{C^{i+1}_{2l+3}}{27C^{i}_{2l+3}} +\frac{1}{2}\right)^4 ,\left(\frac{C^{i+1}_{2l+2}}{27C^{i}_{2l+2}} +\frac{1}{2}\right)^4   \bigg\},
\end{alignat}
we get
\begin{alignat}{1}\label{n4lg3l2}
\begin{split}
(\theta+1)^{2l+5}-\theta^{2l+5}\geq&C^1_{2l+5}t_2^{2l+4}-\frac{C^1_{2l+5}C^1_{2l+4}}{8t_1}t_1^{2l+3}
+C^2_{2l+5}t_2^{2l+3}\\
&-\frac{C^2_{2l+5}C^1_{2l+3}}{8t_1}t_1^{2l+2}+C^3_{2l+5}t_2^{2l+2}-\frac{C^3_{2l+5}C^1_{2l+2}}{8t_1}t_1^{2l+1}.
\end{split}
\end{alignat}
Due to $Q\geq 2275$ when $n=4$, if we choose $k\geq k_2$, where
\begin{alignat}{1}\label{n4k2}
k_2:=\frac{1}{455} \max_{1\leq i\leq 2l+3}\bigg\{ \left( \frac{C^{i+1}_{2l+4}}{2C^{i}_{2l+4}} \right)^{4} , \left( \frac{C^{i+1}_{2l+3}}{2C^{i}_{2l+3}}\right)^{4},
  \left( \frac{C^{i+1}_{2l+2}}{2C^{i}_{2l+2}} \right)^4\bigg\},
\end{alignat}
then
\begin{alignat}{1}\label{n4lg3l1}
\begin{split}
t_2^{2l+4}\geq&\left( \frac{Q}{5} \right)^{\frac{2l+4}{4}}-\frac{C^1_{2l+4}}{2}\left( \frac{Q}{5} \right)^{\frac{2l+3}{4}}
+\frac{C^2_{2l+4}}{4}\left( \frac{Q}{5} \right)^{\frac{2l+2}{4}}-\frac{C^3_{2l+4}}{8}\left( \frac{Q}{5} \right)^{\frac{2l+1}{4}},\\
t_2^{2l+3}\geq&\left( \frac{Q}{5} \right)^{\frac{2l+3}{4}}-\frac{C^1_{2l+3}}{2}\left( \frac{Q}{5} \right)^{\frac{2l+2}{4}}
+\frac{C^2_{2l+3}}{4}\left( \frac{Q}{5} \right)^{\frac{2l+1}{4}}-\frac{C^3_{2l+3}}{8}\left( \frac{Q}{5} \right)^{\frac{2l}{4}},\\
t_2^{2l+2}\geq&\left( \frac{Q}{5} \right)^{\frac{2l+2}{4}}-\frac{C^1_{2l+2}}{2}\left( \frac{Q}{5} \right)^{\frac{2l+1}{4}}
+\frac{C^2_{2l+2}}{4}\left( \frac{Q}{5} \right)^{\frac{2l}{4}}-\frac{C^3_{2l+2}}{8}\left( \frac{Q}{5} \right)^{\frac{2l-1}{4}}.
\end{split}
\end{alignat}
From (\ref{n4lg3l2}) and (\ref{n4lg3l1}), we infer that
\begin{alignat*}{1}
\begin{split}
(\theta+1)^{2l+5}-\theta^{2l+5}\geq&C^1_{2l+5}\left( \frac{Q}{5} \right)^{\frac{2l+4}{4}}+\frac{(2l+5)(l+2)l}{6}\left( \frac{Q}{5}   \right)^{\frac{2l+2}{4}}\\
&-\frac{(4l+7)(2l+5)(2l+3)(l+2)}{24}\left( \frac{Q}{5} \right)^{\frac{2l+1}{4}}\\
&+\frac{(2l+5)(2l+3)(2l-1)(l+2)(l+1)}{24}\left( \frac{Q}{5} \right)^{\frac{2l}{4}}\\
&-\frac{(2l+5)(2l+3)(2l+1)(l+2)(l+1)l}{36}\left( \frac{Q}{5} \right)^{\frac{2l-1}{4}}.
\end{split}
\end{alignat*}

Finally, combing with (\ref{fiolk22}), we get (\ref{4c3}).
\cvd
\end{Proofc}

\begin{Remark}
Using the similar discussion, when $n=5,$ the root of $(\eta+1)^6-\eta^6=Q$ satisfies
\begin{alignat}{1}\label{5t1}
\eta>-\frac{1}{2}+t_1-   \frac{1}{6 t_1}+\frac{11}{720t^3_1}\cdot\frac{99999}{100000}.
\end{alignat}
Similarly, when $n=6$, the root of $(\eta+1)^7-\eta^7=Q$ satisfies
\begin{alignat}{1}\label{6t1}
\eta>-\frac{1}{2}+t_1-\frac{5}{24}\frac{1}{t_1}+\frac{13}{384}\frac{1}{t_1^3}-\frac{1}{222}\frac{1}{t_1^5}
\end{alignat}
where $t_1=\left( Q/(n+1)\right)^{1/n}$.

\end{Remark}

\section{In high dimensional cases}

In this section, we will complete the proof of Theorem \ref{MT11}. First, we aim to find the lower bound of the solution $t$ to the following high order equation for $n\geq5$
\begin{alignat*}{1}
\left(t+\frac{1}{2}\right)^{n+1}-\left(t-\frac{1}{2}\right)^{n+1}=Q.
\end{alignat*}
Obviously the above equation is equals to
\begin{alignat}{1}\label{TT}
(n+1)\left(t^n+\frac{n(n-1)}{24}t^{n-2}+\frac{n(n-1)(n-2)(n-3)}{1920}t^{n-4}  +\cdots\right)=Q.
\end{alignat}
Since $(n+1)t^n\leq Q$, we have $t\leq  ( Q/(n+1) )^{1/n}$. Hence, we can assume
\begin{alignat*}{1}
t=\left(\frac{Q}{n+1}  \right)^{\frac{1}{n}}-C_1.
\end{alignat*}
If we consider the first three terms in (\ref{TT}), then we obtain
\begin{alignat}{1}\label{TT1}
(n+1)\Bigg(t^n+\frac{n(n-1)}{24}t^{n-2}+\frac{n(n-1)(n-2)(n-3)}{1920}t^{n-4}+\cdots  \Bigg)=Q.
\end{alignat}
Putting $t$ into (\ref{TT1}), we get
\begin{alignat*}{1}
\frac{Q}{n+1}= &\frac{Q}{n+1}-\left(\frac{Q}{n+1}  \right)^{\frac{n-1}{n}}nC_1+\frac{n(n-1)}{24}\left(\frac{Q}{n+1}  \right)^{\frac{n-2}{n}}\\
&+\frac{n(n-1)(n-2)(n-3)}{1920}\left(\frac{Q}{n+1}  \right)^{\frac{n-4}{n}}+\frac{n(n-1)}{2}\left(\frac{Q}{n+1}  \right)^{\frac{n-2}{n}}C^2_1\\
&-\frac{n^2(n-1)}{24}\left(\frac{Q}{n+1}  \right)^{\frac{n-3}{n}}C_1+\cdots,
\end{alignat*}
Let $f_1$ be the low order terms of $\left(\frac{Q}{n+1}  \right)^{\frac{n-6}{n}}$ on the right hand of the above equation. Hence, when
\begin{alignat*}{1}
-\left(\frac{Q}{n+1}  \right)^{\frac{n-1}{n}}nC_1+\frac{n(n-1)}{24}\left(\frac{Q}{n+1}  \right)^{\frac{n-2}{n}}\geq 0,
\end{alignat*}
we get
\begin{alignat*}{1}
C_1\leq \frac{n-1}{24}\left(\frac{Q}{n+1}  \right)^{-\frac{1}{n}}.
\end{alignat*}
If we choose $C_1=\frac{n-1}{24}\left(\frac{Q}{n+1}  \right)^{-\frac{1}{n}}$, combing the above inequality with the following formula
\begin{alignat*}{1}
-\frac{n(n-1)^3}{2880} =& \frac{n(n-1)^3}{1920} +\frac{n(n-1)^3}{1152} -\frac{n(n-1)^3}{576}\\
\geq& \frac{n(n-1)(n-2)(n-3)}{1920} +\frac{n(n-1)^3}{1152} -\frac{n^2(n-1)^2}{576},
\end{alignat*}
we have
\begin{alignat*}{1}
0\geq-\frac{n(n-1)^3}{2880}\left(\frac{Q}{n+1}  \right)^{\frac{n-4}{n}}\geq& \frac{n(n-1)(n-2)(n-3)}{1920}\left(\frac{Q}{n+1}  \right)^{\frac{n-4}{n}}\\
&+\frac{n(n-1)}{2}\left(\frac{Q}{n+1}  \right)^{\frac{n-2}{n}}C^2_1\\
&-\frac{n^2(n-1)}{24}\left(\frac{Q}{n+1}  \right)^{\frac{n-3}{n}}C_1 .
\end{alignat*}

Since
\begin{alignat*}{1}
f_1\leq \max_{6\leq i\leq n+1}C^i_{n+1} \cdot\sum^{n+1}_{j=6}\frac{1}{2^{j-1}}\left(\frac{Q}{n+1}  \right)^{\frac{n-j}{n}},
\end{alignat*}
then, under the assumption $k\geq \tilde{k}_1^{\frac{n}{2}}$, where
\begin{alignat}{1}\label{ng7k1}
\tilde{k}_1:=\frac{9375}{149n(n-1)^3}\max_{4\leq i\leq n+1}C^i_{n+1},
\end{alignat}
we conclude that
\begin{alignat*}{1}
0\geq&-\frac{n(n-1)^3}{2880}\left(\frac{Q}{n+1}  \right)^{\frac{n-4}{n}}+f_1.
\end{alignat*}
The above inequality implies
\begin{alignat}{1}\label{rofge}
t=\left(\frac{Q}{n+1}  \right)^{\frac{1}{n}}- \frac{n-1}{24}\left(\frac{Q}{n+1}  \right)^{-\frac{1}{n}}+C_2,
\end{alignat}
where $C_2$ is a positive constant. Putting (\ref{rofge}) into (\ref{TT}), we get
\begin{alignat}{1}\label{c2fg}
\begin{split}
\frac{Q}{n+1}=&R_1+\frac{n(n-1)}{24}R_2+\frac{n(n-1)(n-2)(n-3)}{1920}R_3\\
&+\frac{n(n-1)(n-2)\cdots(n-5)}{322560}R_4+\cdots,
\end{split}
\end{alignat}
where
\begin{alignat*}{1}
R_1:=&\left[\left(\frac{Q}{n+1}  \right)^{\frac{1}{n}}- \frac{n-1}{24}\left(\frac{Q}{n+1}  \right)^{-\frac{1}{n}}+C_2  \right]^n,\\
R_2:=&\left[\left(\frac{Q}{n+1}  \right)^{\frac{1}{n}}- \frac{n-1}{24}\left(\frac{Q}{n+1}  \right)^{-\frac{1}{n}}+C_2  \right]^{n-2},\\
R_3:=&\left[\left(\frac{Q}{n+1}  \right)^{\frac{1}{n}}- \frac{n-1}{24}\left(\frac{Q}{n+1}  \right)^{-\frac{1}{n}}+C_2  \right]^{n-4},\\
R_4:=&\left[\left(\frac{Q}{n+1}  \right)^{\frac{1}{n}}- \frac{n-1}{24}\left(\frac{Q}{n+1}  \right)^{-\frac{1}{n}}+C_2  \right]^{n-6}.
\end{alignat*}
By direct calculation, we get
\begin{alignat*}{1}
R_1=&\frac{Q}{n+1}-\frac{n(n-1)}{24}\left(\frac{Q}{n+1}  \right)^{\frac{n-2}{n}}+\frac{n(n-1)^3}{1152}\left(\frac{Q}{n+1}  \right)^{\frac{n-4}{n}}\\
&-\frac{n(n-1)^4(n-2)}{82944}\left(\frac{Q}{n+1}  \right)^{\frac{n-6}{n}}+n\left(\frac{Q}{n+1}  \right)^{\frac{n-1}{n}}C_2+F_1,\\
R_2=&\left(\frac{Q}{n+1}  \right)^{\frac{n-2}{n}}-\frac{(n-1)(n-2)}{24}\left(\frac{Q}{n+1}  \right)^{\frac{n-4}{n}}+n\left(\frac{Q}{n+1}  \right)^{\frac{n-3}{n}}C_2\\
&+\frac{(n-1)^2(n-2)(n-3)}{1152}\left(\frac{Q}{n+1}  \right)^{\frac{n-6}{n}}+F_2
\end{alignat*}
and
\begin{alignat*}{1}
R_3=&\left(\frac{Q}{n+1}  \right)^{\frac{n-4}{n}}-\frac{(n-1)(n-4)}{24}\left(\frac{Q}{n+1}  \right)^{\frac{n-6}{n}}+F_3,\\
R_4=&\left(\frac{Q}{n+1}  \right)^{\frac{n-6}{n}}+F_4.
\end{alignat*}
where $F_1,F_2,F_3$ and $F_4$ are low order terms of $\left(\frac{Q}{n+1}  \right)^{\frac{n-6}{n}}$. By utilizing the equations of $R_1,R_2,R_3,R_4$ and (\ref{c2fg}), we have
\begin{alignat*}{1}
\begin{split}
\frac{Q}{n+1}=&R_1+\frac{n(n-1)}{24}R_2+\frac{n(n-1)(n-2)(n-3)}{1920}R_3\\
&+\frac{n(n-1)(n-2)(n-3)(n-4)(n-5)}{322560}R_4+\cdots,
\end{split}
\end{alignat*}
which implies
\begin{alignat}{1}\label{c2fg1}
\begin{split}
0=&-\frac{n(n-1)}{24}\left(\frac{Q}{n+1}  \right)^{\frac{n-2}{n}}+\frac{n(n-1)}{24}\left(\frac{Q}{n+1}  \right)^{\frac{n-2}{n}}\\
&+\left(\frac{n(n-1)^3}{1152}-\frac{n(n-1)^2(n-2)}{576}+\frac{n(n-1)\cdots(n-3)}{1920}\right)\left(\frac{Q}{n+1}  \right)^{\frac{n-4}{n}}\\
&+n\left(\frac{Q}{n+1}  \right)^{\frac{n-1}{n}}C_2+(n-1)^2\left(\frac{n\cdots(n-3)}{27648}-\frac{n(n-1)^2(n-2)}{82944} \right)\left(\frac{Q}{n+1}  \right)^{\frac{n-6}{n}}\\
&+\left( \frac{n(n-1)(n-2)\cdots(n-5)}{322560}-\frac{n(n-1)^2(n-2)(n-3)(n-4)}{46080}  \right)  \left(\frac{Q}{n+1}  \right)^{\frac{n-6}{n}}\\
&+\frac{n^2(n-1)}{24}\left(\frac{Q}{n+1}  \right)^{\frac{n-3}{n}}C_2+F_1+\frac{n(n-1)}{24}F_2+\frac{n(n-1)(n-2)(n-3)}{1920}F_3\\
&+\frac{n(n-1)\cdots(n-5)}{322560}F_4+\cdots.
\end{split}
\end{alignat}
Let $f_2$  be  the low order terms of $\left(\frac{Q}{n+1}  \right)^{\frac{n-6}{n}}$ on the right hand of $(\ref{c2fg1})$. Obviously, we have
\begin{alignat*}{1}
f_2\leq \max_{8\leq i\leq n+1}C^i_{n+1} \cdot\sum^{n+1}_{j=8}\frac{1}{2^{j}}\left(\frac{Q}{n+1}  \right)^{\frac{n-j}{n}}.
\end{alignat*}
From the third and fourth terms on the right hand of (\ref{c2fg1}), we conclude that
\begin{alignat*}{1}
\lim_{k\rightarrow\infty} C_2\left(\frac{Q}{n+1}  \right)^{\frac{3}{n}}=\tilde{C}_2,
\end{alignat*}
where $\tilde{C}_2$ is a bounded positive constant.  In fact, when
\begin{alignat*}{1}
C_2+\left(\frac{(n-1)^3}{1152}-\frac{(n-1)^2(n-2)}{576}+\frac{(n-1)(n-2)(n-3)}{1920}\right)\left(\frac{Q}{n+1}  \right)^{\frac{-3}{n}}\leq 0,
\end{alignat*}
we get
\begin{alignat*}{1}
C_2\leq \frac{(n-1)(n-3)(2n+1)}{5760}\left(\frac{Q}{n+1}  \right)^{\frac{-3}{n}}.
\end{alignat*}
If we choose $C_2=\frac{(n-1)(n-3)(2n+1)}{5760}\left(\frac{Q}{n+1}  \right)^{\frac{-3}{n}}$, combining with
\begin{alignat*}{1}
(n-1)^2\left(\frac{n\cdots(n-3)}{27648}-\frac{n(n-1)^2(n-2)}{82944} \right) =-\frac{2n(n-1)\cdots(n-4)(3n-1)}{322560}
\end{alignat*}
and
\begin{alignat*}{1}
\frac{n^2(n-1)}{24}\cdot\frac{(n-1)(n-3)(2n+1)}{5760}=\frac{n^2(n-1)^2(n-3)(2n+1)}{138240},
\end{alignat*}
we get
\begin{alignat}{1}\label{efc2o6}
\begin{split}
\big(r_1&+r_2+r_3\big) \left(\frac{Q}{n+1}  \right)^{\frac{n-6}{n}}\\
=&\frac{n(2n+1)(n-1)(29n^3-116n^2+31n+128)}{2903040} \left(\frac{Q}{n+1}  \right)^{\frac{n-6}{n}},
\end{split}
\end{alignat}
where
\begin{alignat*}{1}
r_1=&\frac{2n(n-1)^3(n-2)(n-4)}{82944},\\
r_2=&-\frac{2n(n-1)\cdots(n-4)(3n-1)}{322560},\\
r_3=&\frac{n^2(n-1)^2(n-3)(2n+1)}{138240}.
\end{alignat*}
Obviously, the right hand of (\ref{efc2o6}) is positive when $n\geq 5$. Therefore, if $k>>n$, then we infer that
\begin{alignat*}{1}
C_2<\frac{(n-1)(n-3)(2n+1)}{5760}\left(\frac{Q}{n+1}  \right)^{\frac{-3}{n}}.
\end{alignat*}
Hence, we can choose
\begin{alignat*}{1}
C_2=\frac{99999}{100000}\frac{(n-1)(n-3)(2n+1)}{5760}\left(\frac{Q}{n+1}  \right)^{\frac{-3}{n}},
\end{alignat*}

Let
\begin{alignat*}{1}
\tilde{c}_1:=&\frac{(n-1)(n-3)(2n+1)}{576000000},\\
\tilde{c}_2:=&\frac{n(2n+1)(n-1)(29n^3-116n^2+31n+128)}{2903040},
\end{alignat*}
when $k\geq\tilde{k}_2^{\frac{n}{2}}$, where
\begin{alignat}{1}\label{ng7k2}
\tilde{k}_2:=\frac{50}{317\tilde{c}_1}\max_{6\leq i\leq n+1}C^i_{n+1},
\end{alignat}
we get
\begin{alignat*}{1}
\tilde{c}_1 \left(\frac{Q}{n+1}  \right)^{\frac{2}{n}}\geq \max_{6\leq i\leq n}C^i_{n+1} \cdot\sum^{n+1}_{j=6}\frac{1}{2^{j-1}}\left(\frac{Q}{n+1}  \right)^{\frac{6-j}{n}} \geq\tilde{c}_2+ \left(\frac{Q}{n+1}  \right)^{\frac{6-n}{n}}f_2.
\end{alignat*}
Therefore, if $k>\max\{\tilde{k}_1^{\frac{n}{2}},\tilde{k}_2^{\frac{n}{2}}\}$, there holds
\begin{alignat}{1}\label{rofge22}
\begin{split}
t\geq &\left(\frac{Q}{n+1}  \right)^{\frac{1}{n}}- \frac{n-1}{24}\left(\frac{Q}{n+1}  \right)^{-\frac{1}{n}}\\
&+\frac{99999}{100000}\frac{(n-1)(n-3)(2n+1)}{5760}\left(\frac{Q}{n+1}  \right)^{\frac{-3}{n}}.
\end{split}
\end{alignat}

Let $\eta=t-\frac{1}{2}$, then $(\eta+1)^{n+1}-\eta^{n+1}=Q$ and
\begin{alignat}{1}\label{7t1}
\begin{split}
\eta>&-\frac{1}{2}+\left(\frac{Q}{n+1}  \right)^{\frac{1}{n}}- \frac{n-1}{24}\left(\frac{Q}{n+1}  \right)^{-\frac{1}{n}}\\
&+\frac{99999}{100000}\frac{(n-1)(n-3)(2n+1)}{5760}\left(\frac{Q}{n+1}  \right)^{\frac{-3}{n}}.
\end{split}
\end{alignat}

For convenience, we define
\begin{alignat*}{1}
\theta=&t_2- \frac{n-1}{24}\bar{t}^{-1},\,\,\,\bar{t}=\left(\frac{Q}{n+1}  \right)^{\frac{1}{n}},\,\,\,t_2=-\frac{1}{2}+\bar{t}.
\end{alignat*}
Therefore, we conclude that
\begin{alignat*}{1}
(\eta+1)^{2l+n+1}-\eta^{2l+n+1}\geq(\theta+1)^{2l+n+1}-\theta^{2l+n+1}.
\end{alignat*}

By direct calculation, we have
\begin{alignat*}{1}
 \theta^{2l+n}=&t_2^{ 2l+n }-\frac{(n-1)C^1_{2l+n}}{24\bar{t}}t_2^{ 2l+n-1 }+\cdots,\\
 \theta^{2l+n-1}=&t_2^{ 2l+n-1 }-\frac{(n-1)C^1_{2l+n-1}}{24\bar{t}}t_2^{ 2l+n-2 }+\cdots,\\
 \theta^{2l+n-2}=&t_2^{ 2l+n-2 }-\frac{(n-1)C^1_{2l+n-2}}{24\bar{t}}t_2^{ 2l+n-3 }+\cdots
\end{alignat*}
and
\begin{alignat*}{1}
t_2^{ 2l+n }=&\bar{t}^{ {2l+n} }-\frac{C^1_{2l+n}}{2}\bar{t}^{ {2l+n-1} }
+\frac{C^2_{2l+n}}{4}\bar{t}^{ {2l+n-2} }-\frac{C^3_{2l+n}}{8}\bar{t}^{ {2l+n-3} }+\cdots,\\
t_2^{2l+n -1}=&\bar{t}^{ {2l+n-1} }-\frac{C^1_{2l+n-1}}{2}\bar{t}^{ {2l+n-2} }
+\frac{C^2_{2l+n-1}}{4}\bar{t}^{ {2l+n-3} }-\frac{C^3_{2l+n-1}}{8}\bar{t}^{ {2l+n-4} }+\cdots,\\
t_2^{ 2l+n-2 }=&\bar{t}^{ {2l+n-2} }-\frac{C^1_{2l+n-2}}{2}\bar{t}^{ {2l+n-3} }
+\frac{C^2_{2l+n-2}}{4}\bar{t}^{ {2l+n-4} }-\frac{C^3_{2l+n-2}}{8}\bar{t}^{ {2l+n-5} }+\cdots.
\end{alignat*}

Moreover, we also have
\begin{alignat*}{1}
(\theta+1)^{2l+n+1}-\theta^{2l+n+1}=&C^1_{2l+n+1} t_2^{ 2l+n }-\frac{(n-1)C^1_{2l+n+1}C^1_{2l+n}}{24\bar{t}}t_2^{ 2l+n-1 }\\
&+C^2_{2l+n+1}t_2^{ 2l+n-1 }-\frac{(n-1)C^2_{2l+n+1}C^1_{2l+n-1}}{24\bar{t}}t_2^{ 2l+n-2 }\\
&+C^3_{2l+n+1}t_2^{ 2l+n-2 }-\frac{(n-1)C^3_{2l+n+1}C^1_{2l+n-2}}{24\bar{t}}t_2^{ 2l+n-3 }\\
&+\cdots.
\end{alignat*}
If we choose $k\geq \tilde{k}_3$, where
\begin{alignat}{1}\label{gnglk3}
\tilde{k}_3:=&\frac{\omega_n^2}{(4\pi)^n}\left( \frac{n+2}{n}\right)^{\frac{n}{2}}\max_{1\leq i\leq 2l+n-1} \{c_{i,2l+n},c_{i,2l+n-1},c_{i,2l+n-2} \}
\end{alignat}
and
\begin{alignat*}{1}
c_{i,2l+n}=& \left(\frac{(n-1)}{24\bar{t}_1}\frac{ C^{i+1}_{2l+n}}{C^i_{2l+n}}+\frac{1}{2} \right)^n ,\,\,\,c_{i,2l+n-1}= \left(\frac{(n-1)}{24\bar{t}_1}\frac{ C^{i+1}_{2l+n-1}}{C^i_{2l+n-1}}+\frac{1}{2} \right)^n ,\\
c_{i,2l+n-2}=& \left(\frac{(n-1)}{24\bar{t}_1}\frac{ C^{i+1}_{2l+n-2}}{C^i_{2l+n-2}}+\frac{1}{2} \right)^n,\,\,\,\bar{t}_1=\frac{4\pi}{\omega_n^{\frac{2}{n}}}\left( \frac{n}{n+2} \right)^{\frac{1}{2}},
\end{alignat*}
then, we obtain the following inequality
\begin{alignat}{1}\label{gngl1}
\begin{split}
(\theta+1)^{2l+n+1}-\theta^{2l+n+1}\geq&C^1_{2l+n+1} t_2^{ 2l+n }-\frac{(n-1)C^1_{2l+n+1}C^1_{2l+n}}{24} \bar{t}^{  2l+n-2  }   \\
&+C^2_{2l+n+1}t_2^{ 2l+n-1 }-\frac{(n-1)C^2_{2l+n+1}C^1_{2l+n-1}}{24} \bar{t}^{ 2l+n-3 }    \\
&+C^3_{2l+n+1}t_2^{ 2l+n-2 }-\frac{(n-1)C^3_{2l+n+1}C^1_{2l+n-2}}{24}\bar{t}^{ 2l+n-4}.
\end{split}
\end{alignat}

Furthermore, if we restrict $k\geq \tilde{k}_4$, where $\tilde{k}_4$ is defined by
\begin{alignat}{1}\label{gnglk4}
&\frac{\omega_n^2}{(4\pi)^n}\left( \frac{n+2}{n}\right)^{\frac{n}{2}}\max_{1\leq i\leq 2l+n-1} \bigg\{ \left( \frac{C^{i+1}_{2l+n}}{2C^i_{2l+n}} \right)^n ,
 \left( \frac{C^{i+1}_{2l+n-1}}{2C^i_{2l+n-1}} \right)^n, \left( \frac{C^{i+1}_{2l+n-2}}{2C^i_{2l+n-2}} \right)^n \bigg\},
\end{alignat}
then, we infer that
\begin{alignat}{1}\label{gngl2}
\begin{split}
t_2^{ 2l+n }\geq&\bar{t}^{ {2l+n} }-\frac{C^1_{2l+n}}{2}\bar{t}^{ {2l+n-1} }
+\frac{C^2_{2l+n}}{4}\bar{t}^{ {2l+n-2} }-\frac{C^3_{2l+n}}{8}\bar{t}^{ {2l+n-3} },\\
t_2^{2l+n -1}\geq&\bar{t}^{ {2l+n-1} }-\frac{C^1_{2l+n-1}}{2}\bar{t}^{ {2l+n-2} }
+\frac{C^2_{2l+n-1}}{4}\bar{t}^{ {2l+n-3} }-\frac{C^3_{2l+n-1}}{8}\bar{t}^{ {2l+n-4} },\\
t_2^{ 2l+n-2 }\geq&\bar{t}^{ {2l+n-2} }-\frac{C^1_{2l+n-2}}{2}\bar{t}^{ {2l+n-3} }
+\frac{C^2_{2l+n-2}}{4}\bar{t}^{ {2l+n-4} }-\frac{C^3_{2l+n-2}}{8}\bar{t}^{ {2l+n-5} }.
\end{split}
\end{alignat}
Putting (\ref{gngl2}) into (\ref{gngl1}), we conclude that
\begin{alignat*}{1}
\begin{split}
(\theta&+1)^{2l+n+1}-\theta^{2l+n+1}\\
\geq&C^1_{2l+n+1}\bar{t}^{ {2l+n} }+\frac{(2l+n+1)(2l+n)l}{12}\bar{t}^{ { 2l+n-2 } }\\
&-\frac{(12l+7n-13)(2l+n+1)(2l+n)(2l+n-1)}{48} \bar{t}^{ { 2l+n-3} }\\
&-\frac{(2l+2n-3)(2l+n+1)(2l+n)(2l+n-1)(2l+n-2)}{96}\bar{t}^{ {2l+n-4} }\\
&-\frac{(2l+n+1)(2l+n)(2l+n-1)(2l+n-2)(2l+n-3)(2l+n-4)}{288}\bar{t}^{ {2l+n-5}}.
\end{split}
\end{alignat*}

According to (\ref{fiolk22}), we complete the proof of Theorem \ref{MT11}.

\section{Lower bounds for the sum of eigenvalues} 

In this section, we will establish the lower bound of $\sum_{i=1}^k\lambda_i$ for any $n$ and $k$. Before we give the proof of Theorem \ref{LWBFAk}, we will give the following key estimate.

\begin{Lemma}\label{KLL}
For $d+q\geq m+2$, where $0\leq m \in \mathbb{N}$ is a constant and $q\geq2$, $s>0$, $\tau>0$, we have the following inequality:
\begin{alignat}{1}\label{KIE}
ds^{d+q}\geq (d+q)s^d\tau^{q}-q\tau^{d+q}+q\sum_{j=0}^m(j+1)s^{j}\tau^{d+q-2-j}(s-\tau)^2.
\end{alignat}
\end{Lemma}

\begin{proof}
Define  function $f(t)$ as
\begin{alignat*}{1}
f(t):=dt^{d+q}-(d+q)t^d+q-q\sum_{l=0}^m(l+1)t^l(t-1)^2,
\end{alignat*}
where $t\geq 0$.

Since (\ref{KIE}) is equivalent to
\begin{alignat*}{1}
ds^{d+q}- (d+q)s^d\tau^{q}+q\tau^{d+q}-q\sum_{l=0}^m(l+1)s^{l}\tau^{d+q-2-l}(s-\tau)^2\geq 0.
\end{alignat*}
If we define $t=\frac{s}{\tau}$, then
\begin{alignat*}{1}
\tau^{d+q}f(t)=ds^{d+q}- (d+q)s^d\tau^{q}+q\tau^{d+q}-q\sum_{l=0}^m(l+1)s^{l}\tau^{d+q-2-l}(s-\tau)^2.
\end{alignat*}
Therefore, if we get $f(t)\geq 0$ for all $t\geq 0$, then we obtain (\ref{KIE}).

Before we show $f(t)\geq 0$ for all $t\geq 0$, we should simplify $f(t)$ as
\begin{alignat*}{1}
f(t)=&dt^{d+q}-(d+q)t^d+q-q\left[(1+m)t^{2+m} +mt^{1+m}-2(1+m)t^{m+1}+1\right]\\
&-q\left[ \sum_{l=1}^{m-1}lt^{1+l}-\sum_{l=1}^{m-1} 2(1+l)t^{1+l} +\sum_{l=1}^{m-1}(l+2)t^{l+1}\right].
\end{alignat*}
On the other hand, we have
\begin{alignat*}{1}
\sum_{l=1}^{m-1}lt^{1+l}-\sum_{l=1}^{m-1} 2(1+l)t^{1+l} +\sum_{l=1}^{m-1}(l+2)t^{l+1}=0.
\end{alignat*}
Hence, we get
\begin{alignat*}{1}
f(t)=& dt^{d+q}-(d+q)t^d+q-q\left[(1+m)t^{2+m} +mt^{1+m}-2(1+m)t^{m+1}+1\right]\\
=&t^{m+1}\left( dt^{d+q-m-1}-(d+q)t^{d-m-1}-q(1+m)t +q (2+m)\right).
\end{alignat*}
Therefore, to prove $f(t)\geq 0$ is equivalent to obtain the following inequality
\begin{alignat*}{1}
dt^{d+q-m-1}-(d+q)t^{d-m-1}-q(1+m)t +q (2+m)\geq 0.
\end{alignat*}

Define $g(t)$ as
\begin{alignat*}{1}
g(t):=dt^{d+q-m-1}-(d+q)t^{d-m-1}-q(1+m)t +q (2+m).
\end{alignat*}
Differentiating $g(t)$ once gives
\begin{alignat*}{1}
g'(t)=d(d+q-m-1)t^{d+q-m-2}-(d+q)(d-m-1)t^{d-m-2}-q (1+m).
\end{alignat*}
Moreover, we differentiate the above formula again to get
\begin{alignat}{1}\label{GG}
\begin{split}
g''(t)=&d(d+q-m-1)(d+q-m-2)t^{d+q-m-3}\\
&-(d+q)(d-m-1)(d-m-2)t^{d-m-3}.
\end{split}
\end{alignat}

Since $g''(t_0)=0$, where
\begin{alignat*}{1}
t_0=\left( \frac{(d+q)(d-m-1)(d-m-2)}{d(d+q-m-1)(d+q-m-2)}   \right)^{\frac{1}{q}}<1.
\end{alignat*}
Hence, when $t\geq t_0$, we have $g''(t)\geq 0$, which implies that $g(t)$ is a convex function on $[0,\infty)$. Thus, we get
\begin{alignat*}{1}
g(t)\geq g(1)+g'(1)(t-1)=0.
\end{alignat*}
Since $g(1)=g'(1)=0$, we know $g(t)\geq 0$ for $t\in[t_0,\infty).$ In particular, $g(t_0)\geq 0$. Next, we should prove that $g(t)$ is non-negative on $[0,t_0)$. In fact, when $0\leq t\leq t_0$, we have $g''(t)\leq 0$ from (\ref{GG}).
Therefore, $g'(t)$ is decreasing on $[0,t_0)$. Consequently, we have
\begin{alignat*}{1}
g(t)\geq g(t_0)\geq 0, \,\,\, \mathrm{on} \,\,\,t\in [0,\infty).
\end{alignat*}
The above inequality implies
\begin{alignat*}{1}
ds^{d+q}&- (d+q)s^d\tau^{q}+q\tau^{d+q}-q\sum_{l=0}^m(l+1)s^{l}\tau^{d+q-2-l}(s-\tau)^2\geq 0.
\end{alignat*}
\end{proof}

The following lemma will be important to obtain our lower bound.

\begin{Lemma}\label{lskl}
Let $n\geq 2$, $\rho>0$, $A>0$ and $2l+n\geq 6$. If $\psi: [0,+\infty)\rightarrow [0,+\infty)$ is a decreasing function (and absolutely continuous) satisfying
\begin{alignat}{1}
-\rho\leq \psi'(s)\leq 0
\end{alignat}
and
\begin{alignat*}{1}
\int_0^\infty s^{n-1}\psi(s)ds=A.
\end{alignat*}
Then
 \begin{alignat*}{1}
\int_0^{\infty}& s^{2l+n-1}\psi(s)ds \\
\geq&\frac{(nA)^{\frac{2l+n}{n}}}{2l+n}\psi(0)^{-\frac{2l}{n}}+\frac{5l(nA)^{\frac{2l+n-2}{n}}}{2n(2l+n)}\psi(0)^{\frac{2n-2l+2}{n}}\rho^{-2}
\\
&-\frac{31l(nA)^{\frac{2l+n-3}{n}}}{9n(2l+n)}\psi(0)^{\frac{3n-2l+3}{n}}\rho^{-3}+\frac{5l(nA)^{\frac{2l+n-4}{n}}}{8n(2l+n)}\psi(0)^{\frac{4n-2l+4}{n}}\rho^{-4}\\
&+\frac{38l(nA)^{\frac{2l+n-6}{n}}}{25n(2l+n)}\psi(0)^{\frac{5n-2l+5}{n}}\rho^{-5}-\frac{317l(nA)^{\frac{2l+n-6}{n}}}{420n(2l+n)}\psi(0)^{\frac{6n-2l+6}{n}}\rho^{-6}.
\end{alignat*}
\end{Lemma}

\begin{proof}
We choose the function $\alpha \psi(\beta t)$ for appropriate $\alpha, \beta >0$, such that $\rho = 1$ and $\psi(0) = 1$. By \cite{M} we can also assume that $B=\int_0^\infty s^{2l+n-1}\psi(s)ds <\infty$. If we let $q(s)=-\psi^{'}(s)$ for $s\geq 0$, we have $0\leq q(s)\leq 1$ and $\int_0^\infty q(s)=\psi(0)=1.$ Moreover, integration by parts implies that
\begin{alignat*}{1}
\int_0^\infty s^{n}q(s)ds=n\int_0^\infty s^{n-1}\psi(s)ds=nA
\end{alignat*}
and
\begin{alignat*}{1}
\int_0^\infty s^{2l+n}q(s)ds\leq (2l+n)B.
\end{alignat*}
Next, let $0\leq a < +\infty$ satisfies that
\begin{alignat}{1}\label{DA}
\int_a^{a+1} s^{n}ds=\int_0^\infty s^{n}q(s)ds=nA.
\end{alignat}
By the same argument as in Lemma 1 of \cite{LY}, such real number $a$ exists. According to Lemma \ref{KLL}, we get
\begin{alignat*}{1}
ns^{2l+n}\geq (2l+n)s^n\tau^{2l}-2l\tau^{2l+n}+2l\sum_{i=0}^m(i+1)s^{i}\tau^{2l+n-2-i}(s-\tau)^2,\,\,\,s\in[a,a+1].
\end{alignat*}
Integrating the both sides of the above inequality, we get
 \begin{alignat*}{1}
n(2l+n)B\geq (2l+n)(nA)\tau^{2l}-2l\tau^{2l+n}+2l\sum_{i=0}^m(i+1)\tau^{2l+n-2-i}\int^{a+1}_as^{i}(s-\tau)^2ds.
\end{alignat*}
Let $\tau=(nA)^{\frac{1}{n}}$.  Notice that
\begin{alignat*}{1}
\int^{a+1}_as^{i}(s-\tau)^2ds=&\sum_{j=0}^iC^j_ia^{i-j}\int^{a+1}_a(s-a)^i(s-\tau)^2ds.
\end{alignat*}
According to the following inequality in \cite{HR}, we get
\begin{alignat*}{1}
2l\sum_{i=0}^m&(i+1)\tau^{2l+n-2-i}\int^{a+1}_as^{i}(s-\tau)^2ds\\
\geq& 2l\sum_{i=0}^m\sum_{j=0}^i(i+1)\tau^{2l+n-2-i} \frac{C^j_ia^{i-j}}{(j+2)^2(j+3)}.
\end{alignat*}
By direct calculation, we have
\begin{alignat*}{1}
2l&\sum_{i=0}^4\sum_{j=0}^i(i+1)\tau^{2l+n-2-i} \frac{C^j_ia^{i-j}}{(j+2)^2(j+3)}\\
=&\frac{l\tau^{2l+n-2}}{6}+\frac{al\tau^{2l+n-3}}{3}+\frac{l\tau^{2l+n-3}}{9}+\frac{a^2l\tau^{2l+n-4}}{2}+\frac{al\tau^{2l+n-4}}{3}+\frac{3l\tau^{2l+n-4}}{40}\\
&+\frac{2la^3\tau^{2l+n-5}}{3}+\frac{2la^2\tau^{2l+n-5}}{3}+\frac{3la\tau^{2l+n-5}}{10}+\frac{4l\tau^{2l+n-5}}{75}+\frac{5la^4\tau^{2l+n-6}}{6}\\
&+\frac{10la^3\tau^{2l+n-6}}{9}+\frac{3la^2\tau^{2l+n-6}}{4}+\frac{4la\tau^{2l+n-6}}{15}+\frac{5l\tau^{2l+n-6}}{126}.
\end{alignat*}
Since
\begin{alignat*}{1}
&\frac{5a^4 }{6}+\frac{10a^3 }{9}+\frac{3a^2 }{4}+\frac{4a }{15}+\frac{5 }{126}\\
&=\frac{5(a+1)^4}{6}-\frac{11(a+1)^3}{9} -\frac{7(a+1)^2}{12}+\frac{53(a+1)}{30}-\frac{317}{420}
\end{alignat*}
and
\begin{alignat*}{1}
\tilde{f}(x)=\frac{5x^4}{6}-\frac{11x^3}{9}-\frac{7x^2}{12}+\frac{53x}{30}-\frac{317}{420}
\end{alignat*}
is increasing on $[0,\infty)$,  we conclude that
\begin{alignat*}{1}
\frac{5a^4 }{6}+\frac{10a^3 }{9}+\frac{3a^2 }{4}+\frac{4a }{15}+\frac{5 }{126}
\geq \tilde{f}(\tau).
\end{alignat*}

Thus, if $2l+n\geq 4$, then we get
\begin{alignat*}{1}
2l&\sum_{i=0}^m(i+1)\tau^{2l+n-2-i}\int^{a+1}_as^{i}(s-\tau)^2ds\\
\geq&\frac{5l}{2}\tau^{2l+n-2}-\frac{31l}{9}\tau^{2l+n-3}+\frac{5l}{8}\tau^{2l+n-4}+\frac{38l}{25}\tau^{2l+n-5}-\frac{317l}{420}\tau^{2l+n-6}
\end{alignat*}
Therefore, we obtain
 \begin{alignat*}{1}
n(2l+n)B\geq&(2l+n)(nA)\tau^{2l}-2l\tau^{2l+n}+\frac{5l}{2}\tau^{2l+n-2}-\frac{31l}{9}\tau^{2l+n-3}\\
&+\frac{5l}{8}\tau^{2l+n-4}+\frac{38l}{25}\tau^{2l+n-5}-\frac{317l}{420}\tau^{2l+n-6},
\end{alignat*}
which implies
 \begin{alignat*}{1}
B\geq&\frac{1}{2l+n}(nA)^{\frac{2l+n}{n}}+\frac{5l}{2n(2l+n)}(nA)^{\frac{2l+n-2}{n}}-\frac{31l}{9n(2l+n)}(nA)^{\frac{2l+n-3}{n}}\\
&+\frac{5l}{8n(2l+n)}(nA)^{\frac{2l+n-4}{n}}+\frac{38l}{25n(2l+n)}(nA)^{\frac{2l+n-6}{n}}-\frac{317l}{420n(2l+n)}(nA)^{\frac{2l+n-6}{n}}.
\end{alignat*}
Thus, we complete the proof of our result.
\end{proof}

According to (\ref{LE}) and the similar discussion of \cite{JX}, we  complete the proof of Theorem \ref{LWBFAk}.

\begin{Remark}
Using the similar discussion of \cite{JX},  if $n\geq m$, we can obtain the following lower bound
\begin{alignat}{1}\label{LEa}
\sum_{j=1}^k \lambda_j^{l}\geq \frac{n}{(2l+n)}\omega_n^{\frac{2l}{n}}\alpha^{\frac{2l}{n}}k^{\frac{2l+n}{n}} +  f(a,l),
\end{alignat}
where $a$ is defined by (\ref{DA}),
\begin{align*}
\begin{split}
f(a,l)&=c_2 \frac{l}{(n+l)}\frac{(m+1)S_{m+3}(a)}{m+3}
\frac{\omega_n^{\frac{m+2-l}{n}}\alpha^{\frac{(m+2)n+2+m-l}{n}}}{\rho^{2+m}}k^{\frac{n+l-2-m}{n}}\\
&\qquad -\frac{l S_{m+2}(a)}{n+l}\frac{\omega_n^{\frac{m+1-l}{d}}\alpha^{\frac{n(m+1)+m+1-l}{n}}}{\rho^{1+m}}k^{\frac{n+l-1-m}{n}},
\end{split}
\end{align*}
$c_2\leq\min\{1,c_1\}$, $S_j(a)=(a+1)^j-a^j$ and
\begin{align*}
c_1=\frac{n(m+1)+m+1-l}{(m+1)[(m+2)n+2+m-l)]}\frac{\beta S_{m+2}(a)}{S_{m+3}(a)}\bigg(\frac{V}{(2\pi)^n} \bigg)^{-\frac{n+1}{n}}\omega_n^{-\frac{1}{n}} k^{\frac{1}{n}}.
\end{align*}
According to \cite{JX},  Theorem \ref{LWBFAk} is a corollary of (\ref{LEa}).

\end{Remark}

~\\
{\bf Acknowledgments}
 This work was supported  by  the National Natural Science Foundation of China, Grant No.12471051,12071424.

\end{document}